\documentclass{article}
\usepackage{amsmath,amssymb,amsthm,graphicx,epsfig,float,url}
\usepackage[colorlinks=true,citecolor={Plum},linkcolor={Periwinkle}]{hyperref}
\usepackage{pdfsync}
\usepackage[usenames, dvipsnames]{xcolor}
\usepackage{tikz}
\usepackage{subfig}
\usepackage{bbm}
\usepackage{stmaryrd}
\usepackage{mathrsfs}  
\usepackage [T1]{fontenc}

\newcommand{\R}{\mathbb{R}}
\newcommand{\N}{\mathbb{N}}

\renewcommand{\geq}{\geqslant}
\renewcommand{\leq}{\leqslant}

\def\B{{\mathbb B}}
\def\e{{\varepsilon}}
\newcommand{\Hun}{\mathbf{(H_1)}}
\newcommand{\Hdeux}{\mathbf{(H_2)}}
\newcommand{\Htrois}{\mathbf{(H_3)}}
\allowdisplaybreaks

\usepackage[dvipsnames]{xcolor}

\newtheorem*{theorem*}{Theorem}

\newtheorem{theorem}{Theorem}  
\newtheorem{proposition}{Proposition}

\newtheorem{definition}{Definition}
\newtheorem{lemma}{Lemma}

\theoremstyle{definition}\newtheorem{remark}{Remark}

\def\O{{\Omega}}
\def\n{{\nabla}}
\def\ur{{u_{\rho}}}
\def\p{{\varphi}}
 \def\pe{{\operatorname{O}(\rho^2)}}

\title{Shape optimization of a Dirichlet type energy for semilinear elliptic partial differential equations}

\author{Antoine Henrot\footnote{Universit\'e de Lorraine, CNRS, Institut Elie Cartan de Lorraine, BP 70239 54506 Vand\oe uvre-l\`es-Nancy Cedex, France ({\tt antoine.henrot@univ-lorraine.fr}).}
	\and Idriss Mazari\footnote{Sorbonne Universit\'es, UPMC Univ Paris 06, UMR 7598, Laboratoire Jacques-Louis Lions, F-75005, Paris, France (\texttt{idriss.mazari@upmc.fr}).}
	\and Yannick Privat\footnote{IRMA, Universit\'e de Strasbourg, CNRS UMR 7501, 7 rue Ren\'e Descartes, 67084 Strasbourg, France ({\tt yannick.privat@unistra.fr}).}
}

\date{}

\begin{document}

\maketitle

\begin{abstract}
Minimizing the so-called ``Dirichlet energy'' with respect to the domain under a volume constraint is a standard problem in shape optimization which is now well understood. This article is devoted to a prototypal non-linear version of the problem, where one aims at minimizing a Dirichlet-type energy involving the solution to a semilinear elliptic PDE with respect to the domain, under a volume constraint. One of the main differences with the standard version of this problem rests upon the fact that the criterion to minimize does not write as the minimum of an energy, and thus most of the usual tools to analyze this problem cannot be used.
By using a relaxed version of this problem, we first prove the existence of optimal shapes under several assumptions on the problem parameters. We then analyze the stability of the ball, expected to be a good candidate for solving the shape optimization problem, when the coefficients of the involved PDE are radially symmetric.  
\end{abstract}

\noindent\textbf{Keywords:} shape optimization, Dirichlet energy, existence/stability of optimal shapes.

\medskip

\noindent\textbf{AMS classification:} 49J45, 49K20.



\section{Introduction}\label{secintro}
\subsection{Motivations and state of the art}
Existence and characterization of domains minimizing or maximizing a given shape functional under constraint is a long story. Such issues have been much studied over the last decades (see e.g. \cite{MR2150214,MR2731611,MR2251558,MR1804683,henrot-pierre}). Recent progress has been made in understanding such issues for problems involving for instance spectral functionals (see e.g. \cite{MR3681143}). 

The issue of minimizing the Dirichlet energy (in the linear case) with respect to the domain is a basic and academical shape optimization problem under PDE constraint, which is by now well understood. This problem reads:

\begin{quote}
\textit{Let $d\in \N^*$ and $D$ be a smooth compact set of $\R^d$. Given $g\in \color{black}W^{-1,2}\color{black}(D)$ and $m\leq |D|$, minimize the Dirichlet energy
$$
J(\Omega)=\frac{1}{2}\int_{\Omega}|\nabla u_\Omega|^2-\langle g,u_\Omega\rangle_{\color{black}W^{-1,2}\color{black}(\Omega),W^{1,2}_0(\Omega)},
$$
where $u_\Omega$ is the unique solution of the Dirichlet problem\footnote{in other words
$$
u_\Omega=\underset{u\in W^{1,2}_0(\color{black}\O\color{black})}{\rm argmin}\left\{ \frac{1}{2}\int_{\Omega}|\nabla u|^2-\langle g,u\rangle_{\color{black}W^{-1,2}\color{black}(\Omega),W^{1,2}_0(\Omega)}\right\}.
$$
} on $\Omega$ associated to $g$, among all open bounded sets $\Omega\subset D$ of Lebesgue measure $|\Omega|\leq m$.}
\end{quote}

As such, this problem is not well-posed and it has been shown (see e.g. \cite{MR1713952} or \cite[Chap. 4]{henrot-pierre} for a survey of results about this problem) that optimal sets only exist within the class
\begin{equation}\label{def:Om}
\mathcal{O}_m=\{\Omega \in \mathcal{A}(D), \ |\Omega|\leq m\},
\end{equation}
where $\mathcal{A}(D)$ denotes the class of quasi-open sets\footnote{Recall that $\Omega\subset D$ is said quasi-open whenever there exists a non-increasing sequence $(\omega_n)_{n\in \N}$ such that 
$$
\forall n\in \N, \ \Omega\cup \omega_n\text{ is open }\quad \text{and}\quad \lim_{n\to +\infty}\operatorname{cap}(\omega_n)=0
$$} of $D$. 

This article is motivated by the observation that, in general, the techniques used to prove existence, regularity and even characterization of optimal shapes for this problem rely on the fact  that the functional is "energetic", in other words that the PDE constraint can be handled by noting that the full shape optimization problem rewrites
$$
\min_{\substack{\Omega \in \mathcal{A}(D)\\ |\Omega|\leq m}}\min_{u\in W^{1,2}_0(D)}\left\{ \frac{1}{2}\int_{\Omega}|\nabla u|^2-\langle g,u\rangle_{\color{black}W^{-1,2}\color{black}(\Omega),W^{1,2}_0(\Omega)}\right\}.
$$
In this article, we introduce and investigate a prototypal problem close to the standard ``Dirichlet energy shape minimization'', involving a nonlinear differential operator. The questions we wish to study here concern existence of optimal shapes and stability issues for ``non energetic'' models.  We note that the literature regarding existence and qualitative properties for  non-energetic, non-linear optimization problems is scarce. We nevertheless mention \cite{MNP}, where existence results are established in certain asymptotic regimes for a shape optimization problem arising in population dynamics.

Since our aim is to investigate the optimization problems in the broadest classes of measurable domains, we consider a volume constraint, which is known to lead to potential difficulties. Indeed, the literature in shape optimization is full of optimization problems that are not well-posed under such constraints.

In the \color{black}perturbed \color{black} version of the Dirichlet problem we will deal with, the linear PDE solved by $u_\Omega$ is changed into a nonlinear one but the functional to minimize remains the same. Since, in such a case, the problem is not  "energetic" anymore (in the sense \color{black} described \color{black} above), the PDE constraint cannot be incorporated into the shape functional. This calls for new tools to be developed in order to overcome this difficulty. Among others, we are interested in the following issues:
\begin{itemize}
\item {\bf Existence:} is the resulting shape optimization problem well-posed?
\item {\bf Stability of optimal sets:} given a minimizer  $\O_0^*$ for the Dirichlet energy in the linear case, is $\O_0^*$ still a minimizer when considering a ``small enough'' non-linear perturbation of the problem?
\end{itemize}

This article is organized as follows: the main results, related to the existence of optimal shapes for Problem \eqref{minJ} and the criticality/stability of the ball are gathered in Section \ref{sec:mainresuults}. Section \ref{sec:prooftheoexists} is dedicated to the proofs of the existence results whereas Section \ref{Se:Shape} is dedicated to the proofs of the stability results. 

\subsection{The shape optimization problem}

In what follows, we  consider a modified version of the problem described above, where the involved PDE constraint is now nonlinear.

\begin{quote}
\textit{Let $d\in \N^*$, $D$ a \color{black} smooth \color{black} compact set of $\R^d$, \color{black}d=2,3\color{black}, $g\in L^2(D)$ and $f\in W^{1,\infty}(\R)$. For a small enough positive parameter $\rho$, let $u_\Omega\in W^{1,2}_0(\Omega)$ be  the unique solution of the problem
\begin{equation}\label{eq:u}
\left\{
\begin{array}{ll}
-\Delta u_{\rho,\O}+\rho f(u_{\rho,\O})=g & \textrm{in }\Omega\\
u_{\rho,\O}\in W^{1,2}_0(\O). &  
\end{array}
\right.
\end{equation}
For $m\leq |D|$, solve the problem:
\begin{equation}\label{minJ}
\inf_{\Omega\in \mathcal{O}_m} J_\rho(\Omega)\quad \text{where }J_\rho(\Omega)=\frac{1}{2}\int_{\Omega}|\nabla u_{\rho,\Omega}|^2-\int_\Omega g u_{\rho,\Omega},
\end{equation}
where $\mathcal{O}_m$ is defined in \eqref{def:Om}.
}
\end{quote}
In this problem, the smallness assumption on the parameter $\rho$ guarantees the well-posedness of the PDE problem \eqref{eq:u} for generic choices of nonlinearities $f$.
\begin{lemma}\label{lem:existFixPt}
There exists $\underline\rho>0$ such that, for any $\O\in \mathcal{O}_m$,  for any $\rho \in [0,\underline \rho)$, Equation \eqref{eq:u}, understood through its variational formulation,  has a unique solution in $W^{1,2}_0(\Omega)$. 
\end{lemma}
This follows from a simple fixed-point argument: let $\lambda_1(\O)$ be the first eigenvalue of the Dirichlet Laplacian on $\O$. We note that the operator
$$
T:\begin{array}[t]{rcl}
W^{1,2}_0(\Omega) & \longrightarrow & W^{1,2}_0(\Omega)\\
u & \longmapsto & w_\Omega,
\end{array}
$$
where $w_\Omega$ is the unique solution of
$$
\left\{
\begin{array}{ll}
-\Delta w-g=-\rho f(u) & \textrm{in }\Omega\\
w\in W^{1,2}_0(\Omega), &  
\end{array}
\right.
$$
is Lipschitz with Lipschitz constant $C_T(\O)$ such that $C_T(\O)\leq \rho \frac1{\lambda_1(\O)}\Vert f\Vert_{W^{1,\infty}}$. By the monotonicity of $\lambda_1$ with respect to domain inclusion (see \cite{henrot2006}), we have, for every $\Omega \in \mathcal{O}_m$, $\lambda_1(D)\leq \lambda_1(\O)$, so that $C_T(\O)\leq \frac{\rho \Vert f\Vert_{W^{1,\infty}}}{\lambda_1(D)}$.
%
%

\section{Main results of the paper}\label{sec:mainresuults}
\subsection{Existence results}

We state hereafter a partial existence result inherited from the linear case. 
Indeed, we will exploit a monotonicity property of the shape functional $J_\rho$ together with its lower-semi continuity for the $\gamma$-convergence to apply the classical  theorem by Buttazzo-DalMaso (see Subsection \ref{Outline}). 
Our approach takes advantage of the analysis of a relaxed formulation of Problem \eqref{minJ}. To introduce it, let us first consider a given box $D\subset \R^n$ (i.e a smooth, compact subset of $\R^n$) such that $|D|> \color{black}m\color{black}$. 

In the minimization problem \eqref{minJ}, let us identify a shape $\Omega$ with its characteristic function $\mathbbm{1}_\Omega$. 
This leads to introducing the ``relaxation'' set
$$
\widehat{\mathcal{O}}_m=\left\{a\in L^\infty(D,[0,1])\text{ such that }\int_D a \leq m\right\}
$$

For a given positive relaxation parameter $M$, we define the (relaxed) functional $\hat{J}_{M,\rho}$ by
\begin{equation}\label{defHatJrho}
\hat J_{M,\rho}(a)=\frac{1}{2}\int_{D}|\nabla u_{M,\rho,a}|^2+\frac{M}{2}\int_{\color{black}D\color{black}} (1-a)u_{M,\rho,a}^2-\int_{\color{black}D\color{black}} g u_{\color{black}M,\rho,a},
\end{equation}
for every $a\in \widehat{\mathcal{O}}_m$, where $u_{M,\rho,a}\in W^{1,2}_0(D)$ denotes the unique solution of the non-linear problem
\begin{equation}\label{eq:ua}
\left\{
\begin{array}{ll}
-\Delta u_{M,\rho,a}+M(1-a)u_{M,\rho,a}+\rho f(u_{M,\rho,a})=g & \textrm{in }D\\
u_{M,\rho,a}\in W^{1,2}_0(D). &  
\end{array}
\right.
\end{equation}
Our existence result involves a careful asymptotic analysis of $u_{M,\rho,a}$ as $\rho \to 0$ to derive a monotonicity property.

Standard elliptic estimates entail that, for every  $ M>0$ and $a\in  \widehat{\mathcal{O}}_m$, one has $u_{M,\rho,a}\in \mathscr C^0(\overline \O).$
\begin{remark}
Such an approximation of $u_{\rho,\Omega}$ is rather standard in the framework of fictitious domains. The introduction of the term $M(1-a)$ in the PDE has an interpretation in terms of porous materials (see e.g. \cite{Evgrafov}) and it may be expected that $u_{M,\rho,a}$ converges in some sense to $u_{\rho,\Omega}$ as $M\to +\infty$ and whenever $a=\mathbbm{1}_\Omega$. This will be confirmed in the analysis to follow.
\end{remark}

Roughly speaking, the existence result stated in what follows requires  the right-hand side of equation \eqref{eq:u} to have  a constant sign. To write the hypothesis down, we need a few notations related to the relaxed problem \eqref{eq:ua}, which is the purpose of the next lemma.
\begin{lemma}\label{claim:borne}
Let $m\in [0,|D|]$, $a\in \widehat{\mathcal{O}}_m$ and $g\in L^2(D)$ be nonnegative.  There exists a positive constant $N_{m,g}$ such that
\begin{equation}
\forall a \in \widehat{\mathcal{O}}_m,\, \forall M>0,\, \forall \rho \in [0,\underline \rho),\quad  \Vert u_{M,\rho,a}\Vert_\infty\leq N_{m,g},
\end{equation}
where $\underline \rho$ is defined in Lemma \ref{lem:existFixPt}, $u_{M,\rho,a}$ denotes the unique solution to \eqref{eq:ua}. In what follows, $N_{m,g}$ will denote the optimal constant in the inequality above, namely
$$
N_{m,g}= \sup \{\Vert u_{M,\rho,a}\color{black}\Vert_\infty\color{black}, \ a \in \widehat{\mathcal{O}}_m,  M>0,  \rho \in [0,\underline \rho)\}.
$$
\end{lemma}

This follows from standard arguments postponed to Section \ref{Annexe:preuveClaim}.

We now state the main results of this section. Let us introduce the assumptions we will consider hereafter:
\begin{itemize}
\item[$\Hun$] There exist two positive numbers $g_0$, $g_1$ such that $g_0<g_1$ and $g_0\leq g(\cdot)\leq g_1$ a.e. in $D$.
\item[$\Hdeux$] One has $f\in W^{1,\infty}(\R)\cap \color{black}\mathcal D^2\color{black}$, where $\color{black}\mathcal D^2\color{black}$ is the set of twice differentiable functions (with second derivatives not necessarily continuous). Moreover, $f(0)\leq 0$ and there exists $\delta >0$ such that the mapping $x\mapsto x f(x)$ is non-decreasing on  $[0, N_{m,g}+\delta]$ where $N_{m,g}$ is given by Lemma \ref{claim:borne}.
\end{itemize}
\begin{theorem}\label{theo:exist}
Let us assume that one of the following assumptions holds true:
\begin{itemize}
\item $g$ or $-g$ satisfies the assumption $\Hun$;
\item $g$ is non-negative and the function $f$ satisfies the assumption $\Hdeux$ or $g$ is non-positive and the function $-f$ satisfies the assumption $\Hdeux$;
\end{itemize}

Then, there exists a positive constant $\rho_0=\rho_0(D,f(0),\Vert f\Vert_{W^{1,\infty}},\color{black} g_0,g_1\color{black})$ such that the shape optimization problem \eqref{minJ} has a solution $\Omega^*$ for every $\rho\in (0,\rho_0)$.\color{black} Furthermore, $|\Omega^*|=m$.\color{black}
\end{theorem}
\begin{remark}
%
The proof of Theorem \ref{theo:exist} rests upon a monotonicity property of the relaxed functional $\hat J_{M,\rho}$ given by \eqref{defHatJrho}. This is the first ingredient that subsequently allows the well-known existence result of Buttazzo and Dal-Maso to be applied.

It is natural to wonder whether or not it would be possible to obtain this  result in a more direct way, for instance by using shape derivatives \color{black} to obtain a monotonicity property. In other words, an idea could be to consider, for a set $E$ whose boundary satisfies minimal regularity assumptions, and for a vector field $V:\partial E\to \R^n$, the shape derivative
$$\lim_{\e\to 0}\frac{J_\rho((\operatorname{Id}+\e V)E)-J_\rho(E)}\e$$ and to prove that this quantity is positive whenever $V\cdot \nu>0$ on $\partial E$. \color{black} We claim that such an approach would require considering domains $\Omega$ \color{black}satisfying a minimum regularity assumption\color{black}, so that the shape derivative (in the sense of Hadamard) of $J_\rho$ at $\Omega$ in direction $V$, where $V$ denotes an adequate vector field, both makes sense and can be written in a workable way (as the integral of the shape gradient multiplied by $V\cdot \nu$). \color{black} We would then need to extend this property to quasi-open sets, as the set of set satisfying such regularity assumptions are not closed for $\gamma$-convergence, which is the natural topology for this class of optimisation problems.\color{black}This relaxed version enables us \color{black}to work with quasi-open sets directly.\color{black}

\color{black}
It is interesting to note that Theorem~\ref{theo:exist} also yields an existence result when restricting ourselves to the set $\tilde{\mathcal O}_m:=\left\{\O\text{ quasi-open, }|\O|=m\right\}$, since Theorem \ref{theo:exist} ensures that, under the appropriate assumption, the optimiser fulfills the volume constraint. 
\color{black}

\end{remark}

\color{black}
We end this section by providing an example where existence within the class of open sets does not hold. It thus shows that it is in general hopeless to get a general existence property for this kind of problem, even by assuming stronger regularity on $f$ and $g$. 
Let us consider the case where $g=0$ and the function $f$ is such that 
\begin{equation}\tag{$\bold{H_4}$}\label{H4}
f(0)<0\quad\text{and}\quad x \mapsto xf(x) \text{ is decreasing.}
\end{equation}
An example of such $f$ is $f(x)=-e^{x^2}.$ In order to make it a globally $W^{1,\infty}(\R)$ function, one can truncate $f$ outside of a large interval $[-M,M]$ and retain Property $(\bold{H_4}$).

\begin{theorem}\label{Pr:NE}
If $g=0$ and $f$ satisfies $(\bold{H_4})$, if the optimization problem \eqref{minJ} has a solution $\O$, then $\O$ has no interior point. In particular, the problème of minimizing $J_\rho(\Omega)$ given by \eqref{minJ} over the set of open domains $\Omega$ such that $|\Omega|\leq m$ has no solution. 
\end{theorem}
\begin{remark}
As will be emphasized in the proof, the key ingredient is that, when $g= 0$ and $f$ satisfies \eqref{H4}, the functional $J_\rho$ is increasing for the inclusion of sets.
\end{remark}
\color{black}


%
%

\subsection{Stability results}  
In what follows, we will work in $\R^2$.\color{black}  We assume that $D$ is large enough so that there exists a centered ball $\B^*$ included in $D$ \color{black} such that $\color{black}|\B^*|=m$\color{black}. We denote by $R>0$ the radius of $\B^*$ and introduce $\mathbb S^*=\partial \B^*$. The notation $\nu$ stands for the outward unit vector on $\mathbb S^*$, in other words $\nu(x)=x/|x|$ for all $x\in \mathbb S^*$.

In this section, we will discuss the local optimality of the ball for small nonlinearities. We will in particular highlight that the local optimality of the ball can be either preserved or lost depending on the choice of the right-hand side $g$. Indeed, if $\rho=0$ and if $g$ is  radially symmetric and non-increasing, the Schwarz rearrangement\footnote{see e.g.  \cite{kawohl} for an introduction to the Schwarz rearrangement.} ensures that, for any $\O\in \mathcal O_m$, $J_0(\O)\geq J_0(\B^*)$.
Without such assumptions, not much is known about the qualitative properties of the optimizers.

According to the considerations above, we will assume in the whole section that 
\begin{itemize}
\item[$\Htrois$]\color{black}We assume that $D$ is a large ball containing $\B^*$, that $g$ is a non-increasing, radially symmetric and non-negative function in $L^2(D)$ and that $f$ is $\mathscr C^2\cap W^{1,\infty}$.\color{black}\end{itemize}
Notice that the analysis to follow can be generalized to sign-changing $g$. Here, this assumption allows us to avoid distinguishing between the cases where the signs of normal derivatives on $\mathbb S^*$ are positive or negative.
For the sake of simplicity, for every $\rho\geq0$, we will call $u_\rho$ the solution of the PDE
\begin{equation}
\left\{\begin{array}{ll}
-\Delta u_\rho+\rho f(u_\rho)=g & \text{ in }\B^*\\
u_\rho\in W^{1,2}_0(\B^*) & \text{ on }\partial \B^*=\mathbb S^*.
\end{array}\right.\end{equation}

Proving a full stationarity result\footnote{ in other words, proving that, for any $\rho\leq \rho^*$, $\B^*$ is the unique minimizer of $J_\rho$ in $\mathcal O_m$} is too intricate to tackle, since we do not know the minimizers topology. 
Hereafter, we investigate the local stability of the ball $\mathbb B^*$: we will prove that the ball is always a critical point, and show that we obtain different stability results, related to the non-negativity of  the second shape derivative of the Lagrangian, depending on $f$ and $g$.

To compute the first and second order shape derivatives, it is convenient to consider vector fields $V\in W^{3,\infty}(\color{black}\R^2\color{black},\R^2)$ and to introduce, for a given admissible vector field $V$ (i.e such that, for $t$ small enough, $(\operatorname{Id}+tV)\B^*\in \mathcal{O}_m$), the mapping
$$f_V:t\mapsto J_\rho \left((\operatorname{Id}+tV)\B^*\right).$$
The first (resp. second) order shape derivative of $J_\rho$ in the direction $V$ is defined as 
$$J_\rho'(\B^*)[V]:=f_V'(0)\, , \text{(resp. $J_\rho''(\B^*)[V,V]:=f_V''(0)$)}.$$ 
To enforce the volume constraint $|\O|=m$, we work with the  unconstrained functional 
$$
\mathcal L_{\Lambda_\rho}:\O\mapsto J_\rho(\O)-\Lambda_\rho \left(\operatorname{Vol}(\O)-m\right),
$$ 
where $\operatorname{Vol}$ denotes the Lebesgue measure in $\R^2$ and $\Lambda_\rho$ denotes a Lagrange multiplier associated with the volume constraint. Recall that, for every domain $\O$ with a $\mathscr C^2$ boundary and every vector field $V\in W^{3,\infty}(\R^2,\R^2)$, we have
$$
\operatorname{Vol}'(\O)[V]=\int_{\partial \O}V\cdot\nu\quad \text{and}\quad \operatorname{Vol}''(\O)[V,V]=\int_{\partial \O}H(V\cdot\nu)^2,
$$ 
where $H$ stands for the mean curvature of $\partial\O$. The local first and second order \color{black}necessary \color{black}  optimality conditions for Problem \eqref{minJ} read as follow:
$$
\left.\begin{array}{r}
\mathcal L_{\Lambda_\rho}'(\O)[V]=0\\ 
\mathcal L_{\Lambda_\rho}''(\O)[V,V]\geq 0 
\end{array}\right\}\text{ for every $V\in W^{3,\infty}(\R^2,\R^2)$ such that }\int_{\mathbb S^*} V\cdot \nu=0.
$$
For further informations about shape derivatives, we refer for instance to \cite[Chapitre 5]{henrot-pierre}. 
Let us state the main result of this section. In what follows, $\rho$  is chosen small enough so that Equation \eqref{eq:u} has a unique solution.

\begin{theorem}\label{Th:Shape}
Let $f$ and $g$ satisfying the assumption $\Htrois$. Let $V\in W^{3,\infty}(\R^2,\R^2)$ denote a vector field such that $\int_{\mathbb S^*} V\cdot \nu=0$.
\begin{enumerate}
\item \textsf{(Shape criticality)} $\B^*$ is a critical shape, in other words $J_\rho'(\B^*)[V]=0$.
\item \textsf{(Shape stability)} Assume that 
\begin{equation}\label{Eq:Hyp}
\color{black}2\color{black}\pi R^2 g(R)\leq \int_{\B^*} g\quad \text{and}\quad 0<\int_{\B^*} g,
\end{equation} 
where $R$ denotes the radius of the ball $\B^*$.
Let $\Lambda_\rho$ be the Lagrange multiplier associated with the volume constraint. There exists $\overline \rho>0$ and $C>0$ such that, for any $\rho\leq \overline \rho$, \begin{equation}\label{Eq:Quanti}
(J_\rho-\Lambda_\rho \operatorname{Vol})''(\B^*)[V,V] \geq C \Vert  V\cdot \nu \Vert^2 _{\color{black}H^\frac12(\O)\color{black}}.\end{equation}
\item \textsf{(Shape instability)} Assume that $g$ is the constant function equal to $1$ and that $f$ is a  non-negative function such that $f'<-1$ on $\left[0,2\Vert u_0\color{black}\Vert_\infty\color{black}\right)$, where $u_0$ is the solution of \eqref{eq:u} with $\rho=0$ and $\O=\mathbb B^*$.
Then, the second order optimality conditions are not fulfilled on $\B^*$: there exists $\overline \rho>0$ and $\hat V\in W^{3,\infty}(\R^2,\R^2)$ such that $\int_{\mathbb S^*} \hat V\cdot \nu=0$ and, for any $\rho\leq \overline \rho$, 
$$
(J_\rho-\Lambda_\rho \operatorname{Vol})''(\B^*)[\hat V,\hat V] <0.
$$
\end{enumerate}
\end{theorem}
\color{black}
\begin{remark}
The coercivity norm obtained in \eqref{Eq:Quanti} could also be obtained in the three-dimensional case, but we only present the proof in the two-dimensional case for the sake of readability. As will be clear throughout the proof, this estimate only relies on the careful use of comparison principles.
\end{remark}
\color{black}
 \begin{remark}
Let us comment on the strategy of proof. It is known that estimates of the kind \eqref{Eq:Quanti} can lead to local quantitative inequalities \cite{DambrineLamboley}. We first establish \eqref{Eq:Quanti} in the case $\rho=0$, and then  extend it to small parameters $\rho$ with the help of a perturbation argument.  
Assumptions of the type \eqref{Eq:Hyp} are fairly well-known, and  amount to requiring that $\B^*$ is a stable shape minimiser \cite{DambrinePierre,henrotpierre}. Finally, the instability result rests upon the following observation: if  $g=1$ and if $V$ is the vector field given by $V(r \cos(\theta),r\sin(\theta))=\cos(\theta)(r \cos(\theta),r\sin(\theta))$, then one has 
$$
(J_0''-\Lambda_0 \operatorname{Vol})''(\B^*)[V,V]=0
$$ 
while higher order modes are stable \cite{DambrinePierre,henrotpierre}. It therefore seems natural to consider such perturbations when dealing with  small parameters $\rho$.
\\It should also be noted that our proof uses a comparison principle, which shortens many otherwise lengthy computations.
\end{remark}
\color{black}
\begin{remark}
The $H^{1/2}$ coercivity norm obtained for the second order shape derivative of the Lagrangian in Estimate \eqref{Eq:Quanti} is the natural one in the framework of shape optimisation, see for instance \cite{DambrineLamboley}. We emphasise that in the case of the functional under scrutiny here, completely explicit computations are not available, but that we obtain this norm through a very careful analysis of the diagonalised shape hessian, using comparison principles.

Although this is not the primary focus of this article, we believe that, with this coercivity property at hand, one can apply the techniques and results of \cite{DambrineLamboley} to derive a local quantitative inequality at the ball. 
\end{remark}
\begin{remark}
The stability result is obtained in the two dimensional case, but could be obtained with the same techniques, provided higher integrability for $g$ holds; indeed, such regularity is needed in fine estimates, see Lemma \ref{Cl:Borne}.

The instability result can readily be extended to higher dimensions, as will follow from the proof which relies on explicit computations on shape derivatives.
\end{remark}
\color{black}

\section{Proof  of Theorem  \ref{theo:exist}}\label{sec:prooftheoexists}
\subsection{General  outline of the proof}\label{Outline}
The  proof of Theorem \ref{theo:exist} rests upon an adaptation of the standard existence result by Buttazzo-DalMaso (see either the original article \cite{ButtazzoDalMaso} or \cite[Thm 4.7.6]{henrot-pierre} for a proof), based on the notion of $\gamma$-convergence, that we recall below.

\begin{definition}For any quasi-open set $\Omega$, let $R_\Omega$ be the resolvent of the Laplace operator on $\Omega$. We say that a sequence of quasi-open sets $(\Omega_k)_{k\in \N}$  in $\mathcal O_m$ $\gamma$-converges to $\O\in \mathcal O_m$   if, for any $\ell \in \color{black}W^{-1,2}\color{black}(D)$, $\left(R_{\Omega_k}(\ell)\right)_{k\in \N}$ converges in $W^{1,2}_0(D)$ to $R_{\Omega}(\ell).$
\end{definition}

The aforementioned existence theorem reads as follows. 

\begin{theorem*}[Buttazzo-DalMaso]
Let $J:\mathcal O_m\rightarrow \R$ be a shape functional satisfying the two following assumptions:
\begin{enumerate}
\item \textsl{(monotonicity)} For every $\Omega_1,\Omega_2\in \mathcal O_m$, 
$\Omega_1\subseteq \Omega_2\Rightarrow J(\Omega_2)\leq J(\Omega_1).$
\item \textsl{($\gamma$-continuity)} $J$ is lover semi-continuous for the $\gamma$-convergence.  
\end{enumerate}
Then the shape optimization problem 
$$\underset{\Omega \in \mathcal O_m}\inf J(\Omega)$$ 
has a solution.
\end{theorem*}

As is customary when using this result, the lower semi-continuity for the $\gamma$-convergence is valid regardless of any sign assumptions on $g$ or of any additional hypothesis on $f$. This is the content of the next result, whose proof is standard and thus, postponed to Appendix~\ref{Annexe:preuveSemicontinuite}.

\begin{proposition}\label{Semicontinuite}
Let $f \in W^{1,\infty}(\R)$ and $\rho \geq 0$. The functional $J_\rho$ is \color{black}continuous \color{black} for the $\gamma$-convergence.
\end{proposition}

It remains hence to investigate the monotonicity of $J_\rho$. Our approach uses a relaxed version of $J_\rho$, namely the functional $\hat J_{M,\rho}$ defined by \eqref{defHatJrho}. More precisely, we will prove under suitable assumptions that
\begin{equation}\label{MonotonieRelaxee}
\forall M\geq 0\, , \forall a_1,a_2\in \widehat{\mathcal O_m}\, , a_1\leq a_2\Longrightarrow \hat J_{M,\rho}(a_1)\geq \hat J_{M,\rho}(a_2).
\end{equation}
\color{black}
It now remains to pass to the limit in \eqref{MonotonieRelaxee} to obtain monotonicity of the functional $J_\rho$.

One could expect, for any $\O\in \mathcal O_m$, that choosing $a=\mathbbm{1}_\O$ and taking the limit $M\to \infty$ would give
$$\hat J_{M,\rho}(\mathbbm{1}_\O)\xrightarrow[M\to +\infty]{} J_\rho(\mathbbm{1}_\O).$$
This is not true in general, but it holds for sets $\O$ that are quasi-stable, see \cite[Chapitre 4]{henrot-pierre}; we recall that a set $\O$ is said to be quasi-stable if, for any $w\in W^{1,2}(D)$, the property ``$w=0$ almost everywhere on $D\backslash \O$'' is equivalent to the property ``$w=0$ quasi-everywhere on $D\backslash \Omega$''. We  underline the fact that, if $\Omega_1$ and $\Omega_2$ are two admissible sets that are equal almost everywhere but not quasi-everywhere, we expect the limits 
$\lim_{M\to \infty}J_{M,\rho}(\mathbbm{1}_{\O_1})$ and $\lim_{M\to \infty}J_{M,\rho}(\mathbbm{1}_{\O_2})$ to be equal. Our strategy is then to first use this relaxation to  prove that the functional $J_\rho$ is monotonous on the set of stable-quasi open sets and then to use the continuity of $J_\rho$ with respect to the $\gamma$-convergence to establish its monotonicity on $\mathcal O_m$.

\paragraph{Using the relaxation for stable quasi-open sets}

The following result, whose proof is postponed to Appendix \ref{append:prooflem:monot2} for the sake of clarity, allows us to make the link between $ \hat J_{M,\rho}$ and  $J_\rho$.
\begin{lemma}\label{lem:monot2}
Let $\Omega \in\mathcal{O}_m$ be a stable quasi-open set. One has
$$
\lim_{M\to +\infty}\hat J_{M,\rho}(\mathbbm{1}_\Omega)=J_\rho(\Omega).
$$
\end{lemma}
Setting then $a_1=\mathbbm{1}_{\Omega_1}\, , a_2=\mathbbm{1}_{\Omega_2}$, 
and passing to the limit in \eqref{MonotonieRelaxee} as $M\to \infty$ gives the monotonicity  of $J_\rho$ on the set 
$$\mathcal O_{m,s}:=\left\{O\in \mathcal O_m\,, \O\text{ is stable}\right\}.$$

\paragraph{Passing from stable quasi-open sets to $\mathcal O_m$} 
The monotonicity of $J_\rho$ on $\mathcal O_m$ is established using the following Lemma, whose proof is postponed to Appendix \ref{Append:Ajout}:
\begin{lemma}\label{Le:Ajout}
If $J_\rho$ is monotonous on $\mathcal O_{m ,s}$, then it is monotonous on $\mathcal O_{m}.$
\end{lemma}
Combining Lemma \ref{Le:Ajout} with Lemma \ref{lem:monot2} and Equation \eqref{MonotonieRelaxee} then gives the required montonicity of the functional $J_\rho$.

\color{black}

\medskip

In the next sections, we will concentrate on showing the monotonicity property \eqref{MonotonieRelaxee}. 
To this aim, we will carefully analyze the so-called ``switching function'' (representing the gradient of the functional $ \hat J_{M,\rho}$) as the parameter $M$ is large enough. 

\subsection{Structure of the switching function}
It is notable that, in this section, we will not make any assumption on $g$ or $f$ \color{black} other than $f\in W^{1,\infty}$ and $g\in W^{-1,2}(D)$.  \color{black}
%
Let $M>0$. Considering the following relaxed version of Problem \eqref{minJ}
 \begin{equation}
 \inf_{a\in \widehat{\mathcal{O}}_m}\hat J_{M,\rho}(a),
 \end{equation}
 it is convenient to introduce the set of {\it admissible perturbations} in view of deriving first order optimality conditions.

\begin{definition}[tangent cone, see e.g. \cite{MR1367820}]
Let $a^*\in \widehat{\mathcal{O}}_m$ and $\mathcal{T}_{a^*}$ be the tangent cone to the set $\widehat{\mathcal{O}}_m$ at $a^*$. The cone $\mathcal{T}_{a^*}$ is the set of functions $h\in L^\infty(D)$ such that, for any sequence of positive real numbers $\varepsilon_n$ decreasing to $0$, there exists a sequence of functions $h_n\in L^\infty(D)$ converging to $h$ for the weak-star topology of $L^\infty(D)$ as $n\rightarrow +\infty$, and $a^*+\varepsilon_nh_n\in \widehat{\mathcal{O}}_m$ for every $n\in\N$.
\end{definition}
In what follows, for any $a\in \widehat{\mathcal{O}}_m$, any element $h$ of the tangent cone  $\mathcal{T}_{a}$ will be called an  {\it admissible direction}.

\begin{lemma}[Differential of $\hat J_{M,\rho}$]\label{lem:cionsOpt1}
Let $a\in \widehat{\mathcal{O}}_m$ and $h\in \mathcal{T}_{a}$. Let $v_{M,\rho,a}$ be the unique solution of
\begin{equation}\label{eq:va}
\left\{
\begin{array}{ll}
-\Delta v_{M,\rho,a}+M(1-a)v_{M,\rho,a}+\rho f'(u_{M,\rho,a})v_{M,\rho,a}=\rho f({u}_{M,a}) & \textrm{in }D\\
v_{M,\rho,a}\in W^{1,2}_0(D). & 
\end{array}
\right.
\end{equation}
Then, $\hat J_{M,\rho}$ is differentiable in the sense of Fr\'echet at $a$ in the direction $h$ and its differential reads
$\langle d\hat J_{M,\rho}(a),h\rangle =\int_D h\Psi_a$, where $\Psi_a$ is the so-called ``switching function'' defined by
$$
\Psi_a=-M\left( v_{M,\rho,a}+\frac{u_{M,\rho,a}}2\right)u_{M,\rho,a}.
$$
\end{lemma}
\begin{proof}[Proof of Lemma \ref{lem:cionsOpt1}]
The Fr\'echet-differentiability of $\hat J_{M,\rho}$ and of the mapping $\mathcal{O}_m\ni a\mapsto u_{M,\rho,a}\in W^{1,2}_0(D)$ at $m^*$ is standard (see e.g. \cite[Chap. 5]{henrot-pierre}). Let us consider an admissible perturbation $h$ of $a$ and let $\dot{u}_{M,\rho,a}$ be the differential of $u_{M,\rho,a}$ at $a$ in direction $h$. One has
\begin{eqnarray*}
\langle d\hat J_{M,\rho}(a),h\rangle &=& \int_D\nabla u_{M,\rho,a}\cdot \nabla \dot{u}_{M,\rho,a}+M\int_D(1-a)u_{M,\rho,a}\dot{u}_{M,\rho,a}-\frac{M}{2}\int_D h u_{M,\rho,a}^2\\
&& -\langle g,\dot{u}_{M,\rho,a}\rangle_{\color{black}W^{-1,2}\color{black}(\Omega),W^{1,2}_0(\Omega)},
\end{eqnarray*}
where $\dot{u}_{M,\rho,a}$ solves the system
\begin{equation}\label{eq:dotua}
\left\{
\begin{array}{ll}
-\Delta \dot{u}_{M,\rho,a}+M(1-a)\dot{u}+\rho f'(u_{M,\rho,a})\dot{u}_{M,\rho,a}=Mh{u}_{M,a} & \textrm{in }D\\
\dot{u}_{M,\rho,a}\in W^{1,2}_0(D). &  
\end{array}
\right.
\end{equation}
Let us multiply the main equation of \eqref{eq:ua} by $\dot u_{M,\rho,a}$ and then integrate by parts. We get
$$
 \int_D\nabla u_{M,\rho,a}\cdot \nabla \dot{u}_{M,\rho,a}+M\int_D(1-a)u_{M,\rho,a}\dot{u}_{M,\rho,a}+\rho \int_D f(u_{M,\rho,a})\dot u_{M,\rho,a} = \langle g,\dot{u}_{M,\rho,a}\rangle_{\color{black}W^{-1,2}\color{black}(\Omega),W^{1,2}_0(\Omega)}
$$
and therefore,
$$
\langle d\hat J_{M,\rho}(a),h\rangle =-\frac{M}{2}\int_D h u_{M,\rho,a^*}^2-\rho \int_D f(u_{M,\rho,a^*})\dot u_{M,\rho,a^*} 
$$
Let us multiply the main equation of \eqref{eq:dotua} by $v_{M,\rho,a}$ and then integrate by parts. We get
$$
 \int_D\nabla v_{M,\rho,a}\cdot \nabla \dot{u}_{M,\rho,a}+M\int_D(1-a)v_{M,\rho,a}\dot{u}_{M,\rho,a}+\rho \int_D f'(u_{M,a})\dot u_{M,a} v_{M,\rho,a}=M\int_Dhu_{M,a}v_{M,\rho,a} .
$$
Similarly, multiplying the main equation of \eqref{eq:va} by $\dot u_{M,\rho,a}$ and then integrating by parts yields
$$
 \int_D\nabla v_{M,\rho,a}\cdot \nabla \dot{u}_{M,\rho,a}+M\int_D(1-a)v_{M,\rho,a}\dot{u}_{M,\rho,a}+\rho \int_D f'(u_{M,\rho,a})\dot u_{M,\rho,a} v_{M,\rho,a}=\rho \int_Df(u_{M,\rho,a})\dot u_{M,\rho,a} .
$$
Combining the two relations above leads to 
$$
\rho \int_Df(u_{M,\rho,a})\dot u_{M,\rho,a}=M\int_Dhu_{M,\rho,a}v_{M,\rho,a}.
$$
Plugging this relation into the expression of $\langle d\hat J_{M,\rho}(a),h\rangle $ above yields the expected conclusion.
\end{proof}

\subsection{Proof that \eqref{MonotonieRelaxee} holds true whenever $\rho$ is small enough}

Let us consider each set of assumptions separately.

\begin{center}
\large{\textsf{Existence under the first assumption: $g$ or $-g$ satisfies the assumption~$\Hun$.}}
\end{center}

According to the discussion carried out in Section \ref{Outline}, proving Theorem \ref{theo:exist} boils down  to proving monotonicity properties for the functional $\hat J_{M,\rho}$ whenever $\rho$ is small enough, which is the purpose of the next result.
\begin{lemma}\label{lem:monot}
Let $a_1$ and $a_2$ be two elements of $\widehat{\mathcal{O}}_m$ such that $a_1\leq a_2$ a.e. in $D$. 
If $g$ or $-g$ satisfies the assumption $\Hun$, then there exists $\rho_1=\rho_1(D,g_0,g_1,\Vert f\Vert_{W^{1,\infty}})>0$ such that
$$
\color{black} \forall M>0\,, \quad\color{black} \rho\in (0,\rho_1)\Rightarrow \hat J_{M,\rho}(a_1)\geq \hat J_{M,\rho}(a_2). 
$$
\end{lemma}

\begin{proof}[Proof of Lemma \ref{lem:monot}]
Assume without loss of generality that $g_0>0$, the case $\color{black} 0\leq g_0\leq -g\leq g_1\color{black}$ being easily inferred by modifying all the signs in the proof below.
Then, one has
$$
-\Delta u_{M,\rho,a}+M(1-a)u_{M,\rho,a}=g-\rho f(u_{M,\rho,a})\geq 0\quad \text{in }D,
$$
whenever $\rho\in (0,g_0/\Vert f\Vert_\infty)$, and therefore, one has $u_{M,\rho,a}\geq 0$ by the comparison principle.

Similarly, notice that
$$
-\Delta u_{M,\rho,a}\leq g_1+\rho \Vert f\Vert_\infty\quad \text{in }D,
$$
which implies that  $u_{M,\rho,a}\leq ( g_1+\rho \Vert f\Vert_\infty) w_D$ were $w_D$ is the torsion function of $D$. By the classical Talenti's estimate of the torsion function \cite{Talenti1979}, we have
$\Vert w_D\Vert_\infty \leq \frac{1}{2d}\left(\frac{|D|}{\omega_d}\right)^{2/d}$ (where $\omega_d$ is the volume of the unit ball). Thus 
\begin{equation}\label{Torsion}\| u_{M,\rho,a}\|_\infty\leq ( g_1+\rho \Vert f\Vert_\infty) \frac{1}{2d}\left(\frac{|D|}{\omega_d}\right)^{2/d}:=C(g_0,\rho,\|f\|_\infty,D).\end{equation}

Setting $U_{M,\rho,a}=\frac{1}{2}u_{M,\rho,a}+ v_{M,\rho,a}$, elementary computations show that $U_{M,\rho,a}$ solves the problem
\begin{equation}\label{eq:U}
\left\{
\begin{array}{ll}
-\Delta U_{M,\rho,a}+\left(M(1-a)+\rho f'(u_{M,\rho,a})\right)U_{M,\rho,a}=\frac{\rho}{2}\left(f(u_{M,\rho,a})+u_{M,\rho,a}f'(u_{M,\rho,a})\right) 
+\frac{g}{2} & \textrm{in }D,\\
U_{M,\rho,a}=0 & \textrm{on }\partial D.
\end{array}
\right.
\end{equation}

\color{black} Before we conclude the proof of Lemma \ref{lem:monot}, we need the following intermediate result on the sign of $U_{M,\rho,a}$.\color{black}
\begin{lemma}\label{Cl:Pos}
Let us choose $\rho_1$ in such a way that
\begin{equation}\label{eq:RhoDelta}
\rho_1(\Vert f\Vert_\infty+C(g_0,\Vert f\Vert_\infty,D) \|f'\|_\infty )< g_0,
\quad \text{and}\quad  
\rho_1 \Vert f'\Vert_\infty\leq \frac{ \lambda_1(D)}2,
\end{equation}
where $\lambda_1(D)$ denotes the first eigenvalue of the Dirichlet-Laplacian operator on $D$ \color{black} and $C(g_0,\Vert f\Vert_\infty,D) $ is given by estimate \eqref{Torsion}. 
For every $\rho\in [0,\rho_1)$, $U_{M,\rho,a}$ is non-negative in $D$.
\end{lemma}
\begin{proof}[Proof of Lemma \ref{Cl:Pos}]
The result follows immediately from the generalized maximum principle  which claims that
if a function $v$ satisfies 
\begin{equation}\label{genmaxpr}
-\Delta v + a(\cdot) v  \geq 0 \quad \mbox{with $a(\cdot)>-\lambda_1(D)$} 
\end{equation}
and $v=0$ on $\partial D$, then $v\geq 0$ a.e. in $D$. \color{black} This is readily seen by multiplying the above inequality by the negative part $v_-$ of $v$ and integrating by part.
\color{black}Here we have chosen $\rho_1$ in such a way that
$$
M(1-a)+\rho f'(u_{M,\rho,a}) \geq -\lambda_1(D)
$$ 
and the right-hand side of \eqref{eq:U}
is non-negative which yields the result.
\end{proof}

Coming back to the proof of Lemma \ref{lem:monot}, consider $h=a_2-a_1$. According to the mean value theorem, there exists $\varepsilon\in (0,1)$ such that
$$
\hat J_{M,\rho}(a_2)-\hat J_{M,\rho}(a_1)=\langle d\hat J_{M,\rho}(a_1+\varepsilon h),h\rangle =-M\int_D hu_{M,a_1+\varepsilon h}U_{M,a_1+\varepsilon h}\leq 0,
$$
according to the combination of the analysis above with Lemma \ref{lem:cionsOpt1}. The expected conclusion follows.
\end{proof}

%

\begin{center}
\large{\textsf{Existence under the second assumption: $g$ is non-negative and the function $f$ satisfies the assumption~$\Hdeux$ or $g$ is non-positive and the function $-f$ satisfies the assumption~$\Hdeux$.}}
\end{center}

The main difference with the previous case is that $g$ might possibly be zero. Deriving the conclusion is therefore trickier and relies on a careful asymptotic analysis of the solution $u_{M,\rho,a}$ as $\rho\to 0$. 
\begin{proposition}\label{DA}
There exists $C=C(D,\Vert f\color{black}\Vert_\infty\color{black})>0$ such that, for any $M\in \R_+$, any $a\in \widehat{\mathcal{O}}_m$, there holds
\begin{equation}
\Vert u_{M,\rho,a}-u_{M,0,a}\Vert_{L^\infty(D)}\leq C\rho.
\end{equation}
\end{proposition}
\begin{proof}
Let us set $z_\rho=u_{M,\rho,a}-u_{M,0,a}$ for any $\rho>0$.
A direct computation yields that $z_\rho$ satisfies
$$-\Delta z_\rho+M(1-a)z_\rho=-\rho f(u_{M,\rho,a}).$$
By comparison with the torsion function $w_D$ of $D$, this implies
$$\|z_\rho\|_\infty \leq \rho \|f\|_\infty \|w_D\|_\infty$$
and the result follows, with a constant $C$ explicit by Talenti's Theorem like in the proof
of Lemma~\ref{lem:monot}.
\end{proof}


\medskip

Let us consider the switching function $\Psi=-MU_{M,\rho,a}u_{M,\rho,a}$ where $u_{M,\rho,a}$ and $U_{M,\rho,a}$ respectively solve \eqref{eq:ua} and \eqref{eq:U}, and we will prove that both $u_{M,\rho,a}$ and $U_{M,\rho,a}$ are non-negative, so that one can conclude similarly to the previous case.

\begin{lemma}
The functions $u_{M,\rho,a}$ and $U_{M,\rho,a}$ \color{black}are \color{black} non-negative whenever $\rho$ is small enough.
\end{lemma}
\begin{proof}
Let us choose $\rho$ such that $\rho \|f\|_\infty <\lambda_1(D)$. Since $u_{M,\rho,a}$ satisfies

\color{black}
$$-\Delta {u_{M,\rho,a}} +M(1-a){u_{M,\rho,a}}+\rho\frac{ f({u_{M,\rho,a}})-f(0)}{u_{M,\rho,a}}u_{M,\rho,a}\geq g-\rho f(0)\geq 0$$ because $f$ satisfies assumption $\Hdeux$.\color{black}

The non-negativity of $u_{M,\rho,a}$ is a consequence of the generalized maximum principle
\eqref{genmaxpr}. Indeed, for $\rho$ small enough, we have 
$$M(1-a)+\rho f(u_{M,\rho,a})>-\lambda_1(D).$$

Since $U_{M,\rho,a}$  satisfies \eqref{eq:U},
the proof follows the same lines assuming the $\rho \|f'\|_\infty < \lambda_1(D)$
and using the assumption $\Hdeux$ to get non-negativity of the right-hand side.
By mimicking the reasoning done at the end of the first case, one gets that \eqref{MonotonieRelaxee} is true if $\rho$ is small enough.
\end{proof}
Thus, in both cases, the monotonicity of the functional is established, so that the theorem of Buttazzo and Dal Maso applies: \color{black} there exists a solution $\O^*\in \mathcal O_m$ of \eqref{minJ}. The fact that $|\O^*|=m$ is a simple consequence of the monotonicity of the functional.\color{black}

\color{black}
\subsection{Proof of Theorem \ref{Pr:NE}: non-existence of regular optimal domains for some $(g,f)$}

Since the proof is mainly based on the use of topological derivatives (\cite{Sokolowski}), we only provide hereafter a sketch of proof. Let us assume the existence of a minimizer $\Omega$ of $J_\rho$ in $\mathcal O_m$ and of an interior point $x_0$ in $\Omega$.
Notice that existence of such a point $x_0$ is not guaranteed for general quasi-open sets, see e.g. \cite[Remark 4.4.7]{Velichkov}.

Let us perform a small circular hole in the domain: define $\Omega_\varepsilon=\Omega\setminus B(x_0,\varepsilon)$ for $\varepsilon>0$ small enough so that $\Omega_\varepsilon\subset \Omega$.

Following \cite{FNRSS,FEIJO2003}, one computes the so-called topological derivative $\left .dJ_\rho(\Omega_\varepsilon)/d\varepsilon \right|_{\varepsilon=0}$. One gets
\begin{equation}\label{TD}
J_\rho(\Omega_\varepsilon)=J_\rho(\Omega)+\pi\e^2 u_{\rho,\O}(x_0)U_{\rho,\O}(x_0) +\operatorname{o}(\varepsilon^2),\end{equation}
where $u_{\rho,\O}$ solves \eqref{eq:u} and $U_{\rho,\O}$ solves
\begin{equation}
\left\{\begin{array}{ll}
-\Delta U_{\rho,\O}+\rho f'(u_{\rho,\O})U_{\rho,\O}=\frac\rho2\left(f(u_{\rho,\O})+u_{\rho,\O}f'(u_{\rho,\O})\right) & \text{in }\Omega\\ 
U_{\rho,\O}\in W^{1,2}_0(\O). & 
\end{array}
\right.
\end{equation}

Since $u_{\rho,\O}$ satisfies 
$$-\Delta u_{\rho,\O}+\rho\frac{f(u_{\rho,\O})-f(0)}{u_{\rho,\O}}u_{\rho,\O}=-\rho f(0)>0$$ according to Assumption~\eqref{H4} and since $x_0$ is an interior point, it follows from the strong maximum principle that one has $u_{\rho,\O}(x_0)>0$.

On the other hand, assumption \eqref{H4} ensures that 
$$-\Delta U_{\rho,\O}+\rho f'(u_{\rho,\O})U_{\rho,\O}<0$$ so that $U_{\rho,\O}(x_0)<0$. As a consequence, for $\e>0$ small enough, we have $$J_\rho(\O_\e)<J_\rho(\O),$$ leading to a contradiction with the minimality of $\Omega$.

\begin{remark}
It is interesting to observe that the asymptotic expansion \eqref{TD} can be formally obtained by using the relaxation method: for a given $M>0$, for $a=\mathbbm 1_\O$ and $h_\e:=-\mathbbm 1_{\B(x_0,\e)}$,  Lemma~\ref{lem:cionsOpt1} yields
$$
\langle d \hat J_{M,\rho}(a),h_\e\rangle=-\int_D h_\e u_{M,\rho,a} U_{M,\rho,a}=\int_{\B(x_0,\e)} u_{M,\rho,a} U_{M,\rho,a} \underset{\e \to 0}\approx \pi \e^2 u_{M,\rho,a}(x_0) U_{M,\rho,a}(x_0).
$$
Passing to the limit $M\to \infty$ provides the expected expression. Of course, such a method is purely formal.

\end{remark}

\color{black}


\section{Proof of Theorem \ref{Th:Shape}}\label{Se:Shape}
Note first that the functional $J_\rho$ is shape differentiable, which follows from standard arguments, see e.g. \cite[Chapitre 5]{henrot-pierre}. 
\\ Our proof of Theorem \ref{Th:Shape} is divided into two steps: after proving the criticality of $\B^*$ for $\rho$ small enough, we compute the second order shape derivative of the Lagrangian associated with the problem at the ball.  Next, we  establish that, under Assumption \eqref{Eq:Hyp}, there exists a positive constant $C_0$ such that, for any admissible $V$, one has
 \begin{equation}\label{Eq:QuantiZero}(J_0-\Lambda_0 \operatorname{Vol})''(\B^*)[V,V]\geq C_0 \Vert V\cdot\nu\Vert _{L^2(\O)}^2.\end{equation}
 Finally, we prove that, for any radially symmetric, non-increasing non-negative $g$,  there exists  $M\in \R$ such that, for any admissible $V$, one has
 \begin{equation}\label{Eq:Convergence}
 (J_\rho-\Lambda_\rho \operatorname{Vol})''(\B^*)[V,V]\geq (J_0-\Lambda_0 \operatorname{Vol})''(\B^*)[V,V]-M\rho \Vert V\cdot\nu\Vert _{L^2(\O)}^2.\end{equation}
Local shape minimality of $\B^*$ for $\rho$ small enough can then be inferred in a straightforward way.

If $V$ is an admissible vector field, we will denote by $u'_{\rho,V}$ and $u''_{\rho,V}$ the first and second order (eulerian) shape derivatives of $u_\rho$ at $\B^*$ with respect to $V$. 
\subsection{Preliminary material}
\begin{lemma}\label{Le:Quali}
Under the assumptions of Theorem \ref{Th:Shape}, i.e when $g$ is radially symmetric and non-increasing function, for $\rho$ small enough, the function $\ur$ is radially symmetric nonincreasing. We write it $\ur=\p_\rho\left(|\cdot|\right)$.
Furthermore, if $\rho=0$, one has
$$
-\frac{\partial u_{0}}{\partial \nu}\geq \frac{R}2 g(R).
$$
\end{lemma}
\begin{proof}[Proof of Lemma \ref{Le:Quali}]
The fact that $\ur$ is a radially symmetric nonincreasing function follows from a simple application of the Schwarz rearrangement. 
Integrating the equation on the ball $\B^*$ yields
$$-\int_{\B^*} \Delta u_0 =- \int_{\partial \B^*} \frac{\partial u_0}{\partial \nu} = -2\pi  \frac{\partial u_0}{\partial \nu}$$
on the one-hand, while using the fact that $g$ is decreasing:
$$-\int_{\B^*} \Delta u_0 = \int_{\B^*} g \geq 2\pi g(R) \int_0^R r dr = \pi R g(R)$$
\end{proof}
By differentiating the main equation \color{black} \eqref{eq:u} with respect to the domain \color{black} and the boundary conditions (see e.g. \cite[Chapitre 5]{henrot-pierre}), we get that the functions $u'_{\rho,V}$ and $u''_{\rho,V}$ satisfy
\begin{equation}\label{Eq:FSD}
\left\{\begin{array}{ll}
-\Delta u'_{\rho,V}+\rho f'\left(\ur\right)u'_{\rho,V}=0& \text{ in }\B^*\\
u'_{\rho,V}=- \frac{\partial \ur}{\partial \nu }V\cdot \nu&\text{ on }\partial \B^*
\end{array}\right.\end{equation}
and 
\begin{equation}\label{Eq:SSD}
\left\{\begin{array}{ll}
-\Delta u''_{\rho,V}+\rho f'\left(\ur\right)u''_{\rho,V}+\rho f''(\ur)\left(u'_{\rho,V}\right)^2=0& \text{ in }\B^*\\
u''_{\rho,V}=-2 \frac{\partial u'_{\rho,V}}{\partial \nu }V\cdot \nu-(V\cdot\nu)^2\frac{\partial^2 \ur}{\partial\nu^2}&\text{ on }\partial \B^*.
\end{array}\right.
\end{equation}

 \subsection{Proof of the shape criticality of the ball}
 Proving the shape criticality of the ball boils down to showing the existence of a Lagrange multiplier $\Lambda_\rho\in \R$ such that for every admissible vector field $V\in W^{3,\infty}(\R^2,\R^2)$, one has
 \begin{equation}\label{eqMetz1621}
 (J_\rho-\Lambda_\rho \operatorname{Vol})'(\B^*)[V]=0
 \end{equation}
 Standard computations (see e.g. \cite[chapitre 5]{henrot-pierre})
%
yield
\begin{align*}
J_\rho'(\B^*)[V]&=\int_{\B^*}\langle \n \ur,\n u'_{\rho,V}\rangle-\int_{\B^*}gu'_{\rho,V}+\int_{\mathbb S^*}\frac12|\n \ur|^2V\cdot\nu
\\&=\int_{\mathbb S^*} u'_{\rho,V} \frac{\partial \ur}{\partial \nu}-\rho \int_{\B^*}u'_{\rho,V}f(\ur)+\int_{\mathbb S^*}\frac12|\n \ur|^2V\cdot\nu
\\&=-\int_{\mathbb S^*} \left( \frac{\partial \ur}{\partial \nu}\right)^2V\cdot\nu+\int_{\mathbb S^*}\frac12|\n \ur|^2V\cdot\nu-\rho \int_{\B^*}u'_{\rho,V}f(\ur)
\\&=-\frac12\int_{\mathbb S^*}|\n \ur|^2V\cdot\nu-\rho \int_{\B^*}u'_{\rho,V}f(\ur).
\end{align*}
We introduce the adjoint state $p_\rho$ as the unique solution of 
\begin{equation}\label{Eq:Adj}
\left\{\begin{array}{ll}
-\Delta p_\rho+\rho p_\rho f'(\ur)+\rho f(\ur)=0& \text{ in }\B^*\\
p_\rho=0 & \text{ on }\mathbb S^*.
\end{array}\right.
\end{equation} 
Since $\ur$ is radially symmetric, so is $p_\rho$. Multiplying the main equation of \eqref{Eq:Adj} by $u'_{\rho,V}$ and integrating by parts yields
$$
-\rho \int_{\B^*}u'_{\rho,V}f(\ur)=\int_{\mathbb S^*} \frac{\partial p_\rho}{\partial \nu}\frac{\partial \ur}{\partial \nu}V\cdot \nu,
$$ 
and finally
$$
J_\rho'(\B^*)[V]=\int_{\mathbb S^*}  \left(\frac{\partial p_\rho}{\partial \nu}\frac{\partial \ur}{\partial \nu}-\frac12\left(\frac{\partial \ur}{\partial \nu}\right)^2\right)V\cdot \nu.
$$
Observe that $\frac{\partial p_\rho}{\partial \nu}$ and $\frac{\partial \ur}{\partial \nu}$ are constant on $\mathbb S^*$since $\ur$ and $p_\rho$ are radially symmetric. Introduce the real number $\Lambda_\rho$ given by
\begin{equation}\label{def:Lambrho}
\Lambda_\rho= \left.\frac{\partial p_\rho}{\partial \nu}\frac{\partial \ur}{\partial \nu}-\frac12\left(\frac{\partial \ur}{\partial \nu}\right)^2\right|_{\mathbb S^*},
\end{equation}
we get that \eqref{eqMetz1621} is satisfied, whence the result.

\medskip

In what follows, we will exploit the fact that the adjoint state is radially symmetric. In the following definition, we sum-up the notations we will use in what follows.
 
\begin{definition}\label{def:fctsrad}
Recall that $\varphi_\rho$ (defined in Lemma \ref{Le:Quali}) is such that
$$
u_\rho (x)=\varphi_\rho(|x|), \quad \forall x\in \mathbb B^*.
$$
Since $p_\rho$ is also radially symmetric,  introduce $\phi_\rho$ such that
$$
p_\rho(x)=\phi_\rho(|x|), \quad \forall x\in \mathbb B^*.
$$
\end{definition}

%
%

\subsection{Second order optimality conditions}
Let us focus on the second and third points of Theorem \ref{Th:Shape}, especially on \eqref{Eq:Quanti}. Since $\B^*$ is a critical shape, it is enough to work with normal vector fields, in other words vector fields $V$ such that $V=(V\cdot\nu) \nu$ on $\mathbb S^*$. Consider such a vector field $V$. For the sake of notational simplicity, let us set $J_\rho''=J_\rho''(\B^*)[V,V]$, $\mathcal L_{\Lambda_\rho}''=(J_\rho-\Lambda_\rho \operatorname{Vol})''(\B^*)[V,V]$, $u=\ur$, $u'=u'_{\rho,V}$ and $u''=u''_{\rho,V}$. 
\color{black}\subsubsection{Computation of the second order derivative at the ball}\color{black}
To compute the second order derivative, we use the Hadamard second order formula \cite[Chap.~5, p.~227]{henrot-pierre} for normal vector fields, namely\color{black}
 $$\left.\frac{d^2}{dt^2}\right|_{t=0}\int_{(\operatorname{Id}+tV)\B^*} k(t)=\int_{\B^*} k''(0)+2\int_{\mathbb S^*} k'(0) V\cdot \nu+\int_{\mathbb S^*}\left(\frac1R k(0)+\frac{\partial k(0)}{\partial \nu}\right)(V\cdot\nu)^2,
 $$
applied to $k(t)=\frac12|\n u_t|^2-gu_t$, \color{black} where $u_t$ denotes the solution of \eqref{eq:u} on $(\operatorname{Id}+tV)\B^*$.

The Hadamard formula along with the weak formulation of Equations \eqref{Eq:FSD}-\eqref{Eq:SSD} yields
 \begin{align*}
 J_\rho''&=\int_{\B^*} \langle \n u,\n u''\rangle-\int_{\B^*} gu''+\int_{\B^*}|\n u'|^2+2\int_{\mathbb S^*} \frac{\partial u}{\partial \nu}\frac{\partial u'}{\partial \nu}V\cdot\nu-2\int_{\mathbb S^*}gu'V\cdot\nu
 \\&+\int_{\mathbb S^*}\left(\frac1{2R}|\n u|^2+\frac{\partial u}{\partial \nu}\frac{\partial^2 u}{\partial \nu^2}-g\frac{\partial u}{\partial \nu}\right)(V\cdot\nu)^2
 \\&=-\rho \int_{\B^*} f(u)u''-\rho \int_{\B^*} (u')^2f'(u)+\int_{\mathbb S^*}u''\frac{\partial u}{\partial \nu}+\int_{\mathbb S^*}u'\frac{\partial u'}{\partial \nu} 
 \\&+2\int_{\mathbb S^*}\frac{\partial u}{\partial \nu}\frac{\partial u'}{\partial \nu}V\cdot\nu-2\int_{\mathbb S^*} gu'V\cdot\nu
 +\int_{\mathbb S^*}\left(\frac1{2R}|\n u|^2+\frac{\partial u}{\partial \nu}\frac{\partial^2 u}{\partial \nu^2}-g\frac{\partial u}{\partial \nu}\right)(V\cdot\nu)^2
 \\&=-\rho \int_{\B^*} f(u)u''-\rho \int_{\B^*} (u')^2f'(u)+\int_{\mathbb S^*}\left(-2 \frac{\partial u'}{\partial \nu }V\cdot \nu-\frac{\partial^2 u}{\partial \nu^2} (V\cdot\nu)^2\right)\frac{\partial u}{\partial \nu}
 \\&-\int_{\mathbb S^*}\frac{\partial u}{\partial \nu}\frac{\partial u'}{\partial \nu} V\cdot\nu+2\int_{\mathbb S^*}\left(\frac{\partial u}{\partial \nu}\frac{\partial u'}{\partial \nu}- gu'\right)V\cdot \nu
 +\int_{\mathbb S^*}\left(\frac1{2R}|\n u|^2+\frac{\partial u}{\partial \nu}\frac{\partial^2 u}{\partial \nu^2}-g\frac{\partial u}{\partial \nu}\right)(V\cdot\nu)^2
  \\&=-\rho \int_{\B^*} f(u)u''-\rho \int_{\B^*} (u')^2f'(u)-\int_{\mathbb S^*}\frac{\partial u}{\partial \nu}\frac{\partial^2 u}{\partial \nu^2}(V\cdot\nu)^2-\int_{\mathbb S^*}\frac{\partial u}{\partial \nu}\frac{\partial u'}{\partial\nu}(V\cdot\nu)
  \\&+2\int_{\mathbb S^*}g \frac{\partial u}{\partial \nu}(V\cdot\nu)^2+\int_{\mathbb S^*}\left(\frac1{2R}|\n u|^2+\frac{\partial u}{\partial \nu}\frac{\partial^2u}{\partial \nu^2}-g\frac{\partial u}{\partial \nu}\right)(V\cdot\nu)^2
  \\&=-\rho \int_{\B^*} f(u)u''-\rho \int_{\B^*} (u')^2f'(u)+\int_{\mathbb S^*}\left(\frac1{2R} \left(\frac{\partial u}{\partial \nu}\right)^2+g\frac{\partial u}{\partial \nu}\right)(V\cdot\nu)^2-\int_{\mathbb S^*} \frac{\partial u}{\partial \nu}\frac{\partial u'}{\partial\nu}V\cdot \nu.
 \end{align*}
 
As such, the two first terms of the sum in the expression above are not tractable. Let us rewrite them. 
Multiplying the main equation of \eqref{Eq:Adj} by $u''$ and integrating two times by parts yields
 $$
 -\rho \int_{\B^*} f(u)u''=\int_{\mathbb S^*}u'' \frac{\partial p_\rho}{\partial \nu}-\rho \int_{\B^*} (u')^2p_\rho f''(u).
 $$
 To handle the last term of the right-hand side, let us introduce the function $\lambda_\rho$ defined as the solution of
  \begin{equation}\label{Eq:Lambda}
 \left\{
 \begin{array}{ll}
 -\Delta \lambda_\rho+\rho \lambda_\rho f' (u)+\rho u' p_\rho f''(u)=0 & \text{ in }\B^*\\
 \lambda_\rho=0 & \text{ on }\mathbb S^*.\end{array}\right.
 \end{equation}
Multiplying this equation by $u'$ and integrating by parts gives
$$-\rho\int_{\mathbb B^*} f''(u)(u')^2=\int_{\mathbb S^*}u'\frac{\partial \lambda_\rho}{\partial \nu}.$$

 To handle the term $-\rho \int_{\B^*} (u')^2f'(u)$ of $J_\rho''$, we introduce the function $\eta_\rho$, defined as the only solution to 
 \begin{equation}\label{Eq:Eta}
 \left\{
 \begin{array}{ll}
 -\Delta \eta_\rho+\rho \eta_\rho f' (u)+\rho u' f'(u)=0 & \text{ in }\B^*\\
 \eta_\rho=0 & \text{ on }\mathbb S^*.
 \end{array}\right.
 \end{equation}
 Multiplying this equation by $u'$ and integrating by parts gives
 $$-\rho \int_{\B^*} (u')^2f'(u)=\int_{\mathbb S^*}u' \frac{\partial \eta_\rho}{\partial \nu}=-\int_{\mathbb S^*} V\cdot\nu \frac{\partial \eta_\rho}{\partial \nu}\frac{\partial u}{\partial \nu}.
 $$ 
 Gathering these terms, we have
 \begin{align*}
 J_\rho''&=\int_{\mathbb S^*}u'' \frac{\partial p_\rho}{\partial \nu}+\int_{\mathbb S^*}u'\frac{\partial \lambda_\rho}{\partial \nu}-\int_{\mathbb S^*} \frac{\partial \eta_\rho}{\partial \nu}\frac{\partial u}{\partial \nu}V\cdot\nu -\int_{\mathbb S^*} \frac{\partial u}{\partial \nu}\frac{\partial u'}{\partial\nu}V\cdot \nu \\
 &+\int_{\mathbb S^*}\left(\frac1{2R} \left(\frac{\partial u}{\partial \nu}\right)^2+g\frac{\partial u}{\partial \nu}\right)(V\cdot\nu)^2.
 \end{align*} 
 Using that 
 $$
 \Lambda_\rho=\left.\frac{\partial p_\rho}{\partial \nu}\frac{\partial \ur}{\partial \nu}-\frac12\left(\frac{\partial \ur}{\partial \nu}\right)^2\right|_{\mathbb S^*}\quad \text{and}\quad \operatorname{Vol}''(\B^*)=\int_{\mathbb S^*}\frac1R  (V\cdot \nu) ^2,
 $$ 
one computes
\begin{equation}\label{Eq:OptSecond}
\boxed{\begin{split}
\mathcal L_{\Lambda_\rho}''& =\int_{\mathbb S^*}u'' \frac{\partial p_\rho}{\partial \nu}+\int_{\mathbb S^*}u'\frac{\partial \lambda_\rho}{\partial \nu}-\int_{\mathbb S^*} \frac{\partial \eta_\rho}{\partial \nu}\frac{\partial u}{\partial \nu}V\cdot\nu-\int_{\mathbb S^*} \frac{\partial u}{\partial \nu}\frac{\partial u'}{\partial\nu}V\cdot \nu 
 \\&+\int_{\mathbb S^*}\left(-\frac{\Lambda_\rho}R+\frac1{2R} \left(\frac{\partial u}{\partial \nu}\right)^2+g\frac{\partial u}{\partial \nu}\right)(V\cdot\nu)^2
\end{split}}
\end{equation}

 \subsubsection{Expansion in Fourier Series}
In this section, we recast the expression of $\mathcal L_{\Lambda_\rho}''$ in a more tractable form, by using the method introduced by Lord Rayleigh: since we are dealing with vector fields normal to $\mathbb S^*$, we expand $V\cdot\nu$ as a Fourier series. This leads to introduce the sequences of Fourier coefficients $(\alpha_k)_{k\in \N^*}$ and $(\beta_k)_{k\in \N^*}$ defined by:
 $$
  V\cdot \nu=\sum_{k\in \N^*}\Big(\alpha_k \cos(k\cdot)+\beta_k\sin(k\cdot)\Big),
  $$
  the equality above being understood in a $L^2(\mathbb S^*)$ sense. 
  
 Let $v_{k,\rho}$ (resp. $w_{k,\rho}$) denote the function $u'$ associated to the perturbation choice $V_k$ given by 
 $V_k=V_k^c:= \cos(k\cdot)\nu$ (resp. $V_k=V_k^s :=\sin(k\cdot)\nu$), in other words, $v_{k,\rho}=u'_{\rho,V_k^c}$ (resp. $w_{k,\rho}=u'_{\rho,V_k^s}$).
 Then, one shows easily (by uniqueness of the solutions of the considered PDEs) that for every $k \in \N$, there holds 
 $$
 v_{k,\rho}(r,\theta)=\psi_{k,\rho}(r) \cos(k\theta) \text{ (resp. $w_{k,\rho}(r,\theta)=\psi_{k,\rho}(r)\sin(k\theta)$)},
 $$
 where $(r,\theta)$ denote the polar coordinates in $\R^2$, where $\psi_{k,\rho}$ solves
\begin{equation}\label{def:psikrho}
\left\{\begin{array}{ll}
-\frac1r(r\psi_{k,\rho}')'=-\left(\frac{k^2}{r^2}+\rho f'(u)\right)\psi_{k,\rho} & \text{ in }(0,R)\\ 
\psi_{k,\rho}(R)=-\p_{\rho}'(R). &
\end{array}\right.
\end{equation}
By linearity, we infer that 
$$u'=\sum_{k\in \N^*}\alpha_k v_{k,\rho}+\beta_k w_{k,\rho}.$$
For every $k\in \N^*$, let us introduce $\eta_{k,\rho}$ as the solution of \eqref{Eq:Eta} associated with $v_{k,\rho}$. One shows that $\eta_{k,\rho}$ satisfies
\begin{equation}\label{Eq:EtaK}
 \left\{
 \begin{array}{ll}
 -\Delta \eta_{k,\rho}+\rho  f' (u)\eta_{k,\rho}+\rho  f'(u)v_{k,\rho}=0 & \text{ in }\B^*\\
 \eta_{k,\rho}=0 & \text{ on }\mathbb S^*.
 \end{array}\right.
 \end{equation}
 Similarly, one shows easily that 
 $$
 \eta_{k,\rho}(r,\theta)=\xi_{k,\rho}(r)\cos(k\theta),
 $$
 where $\xi_{k,\rho}$ satisfies
 \begin{equation}\label{def:xikrho}
\left\{\begin{array}{ll}
-\frac1r(r\xi_{k,\rho}')'=-\left(\frac{k^2}{r^2}+\rho f'(u)\right)\xi_{k,\rho} -\rho \psi_{k,\rho} & \text{ in }(0,R)\\ 
\xi_{k,\rho}(R)=0.&
\end{array}\right.
\end{equation}
Notice that one has $\xi_{k,\rho}=0$ whenever  $\rho=0$, which can be derived obviously from \eqref{Eq:Eta}.

\color{black}We recall that $u_\rho$ is radially symmetric and that we denote by $r\mapsto \varphi_\rho(r)$ this radial function.\color{black}

Finally, we introduce a last set of equations related to $\lambda_\rho$. Let us define $\zeta_{k,\rho}$ as the solution of 
\begin{equation}\label{def:zetakrho}
\left\{\begin{array}{ll}
-(r\zeta_{k,\rho}')'=-\frac{k^2}{r^2}\zeta_{k,\rho}-r\rho \zeta_{k,\rho}f'(u)-\rho r \psi_{k,\rho}\phi_{\rho}f''(u) & \text{ in }(0,R)\\
 \zeta_{k,\rho}(R)=0. & 
\end{array}\right.
\end{equation}
and verify that $\lambda_{\rho}=\zeta_{k,\rho}(r)\cos(k\theta)$ whenever $V=V_k$.

\begin{proposition}\label{Pr:Sep} 
The quadratic form $\mathcal L_{\Lambda_\rho}''$ expands as
\begin{equation}
\mathcal L_{\Lambda_\rho}''=\sum_{k=1}^\infty \omega_{k,\rho}\left(\alpha_k^2+\beta_k^2\right),\end{equation}
where, for any $k\in \N^*$, 
\begin{multline}\label{def:omegarhok}
\omega_{k,\rho}=\pi R\Big(-2\psi_{k,\rho}'(R)\phi_{\rho}'(R)-\p_{\rho}'(R)\zeta_{k,\rho}'(R)-\p_{\rho}''(R)\phi_{\rho}'(R) \\
-\xi_{k,\rho}'(R)\p_{\rho}'(R)-\frac{\Lambda_\rho}R+\frac1{2R} (\p_{\rho}')^2+g(R)\p_{\rho}'(R)-\p_{\rho}'(R)\psi_{k,\rho}'(R)\Big),
\end{multline}
the functions  $\psi_{k,\rho}$, $\xi_{k,\rho}$, $\zeta_{k,\rho}$ being  respectively   defined by \eqref{def:psikrho}, \eqref{def:xikrho}, \eqref{def:zetakrho}, and $\Lambda_\rho$ is given by \eqref{def:Lambrho}.
\end{proposition}
\begin{proof}[Proof of Proposition \ref{Pr:Sep}]
Let us first deal with the particular case $V\cdot\nu=\cos(k\cdot)$. According to \eqref{Eq:FSD}, \eqref{Eq:SSD} and \eqref{Eq:OptSecond}, one has
\begin{eqnarray*}
\mathcal L_{\Lambda_\rho}'' &=& \int_{\mathbb S^*}u'' \frac{\partial p_\rho}{\partial \nu}+\int_{\mathbb S^*}u'\frac{\partial \lambda_\rho}{\partial \nu}-\int_{\mathbb S^*} \frac{\partial \eta_\rho}{\partial \nu}\frac{\partial u}{\partial \nu}V\cdot\nu-\int_{\mathbb S^*} \frac{\partial u}{\partial \nu}\frac{\partial u'}{\partial\nu}V\cdot \nu 
 \\
 &&+\int_{\mathbb S^*}\left(-\frac{\Lambda_\rho}R+\frac1{2R} \left(\frac{\partial u}{\partial \nu}\right)^2+g\frac{\partial u}{\partial \nu}\right)(V\cdot\nu)^2 
 \\
 &=& R\int_0^{2\pi}\left(-2\cos(k\theta)^2\psi_{k,\rho}'(R)-\cos(k\theta)^2 \p_{\rho}''(R)\right)\phi_\rho'(R)d\theta
 -R\int_0^{2\pi} \cos(k\theta)^2 \p_{\rho}'(R)\zeta_{k,\rho}'(R)d\theta
 \\
 && -R\int_0^{2\pi}\cos(k\theta)^2\xi_{k,\rho}'(R)\p_{\rho}'(R)d\theta
 +R\int_0^{2\pi}\cos(k\theta)^2\left(-\frac{\Lambda_\rho}R+\frac1{2R} (\p_{\rho}')^2+g(R)\p_{\rho}'(R)\right)d\theta
 \\ 
 &&
 -R\int_0^{2\pi}\cos(k\theta)^2\p_{\rho}'(R)\psi_{k,\rho}'(R)d\theta
 \end{eqnarray*}
 and therefore
 \begin{eqnarray*}
\frac{\mathcal L_{\Lambda_\rho}''}{\pi R} &=& -2\psi_{k,\rho}'(R)\phi_{\rho}'(R)-\p_{\rho}'(R)\zeta_{k,\rho}'(R)-\p_{\rho}''(R)\phi_{\rho}'(R) -\xi_{k,\rho}'(R)\p_{\rho}'(R)
 \\&&-\frac{\Lambda_\rho}R+\frac1{2R} (\p_{\rho}')^2+g(R)\p_{\rho}'(R)-\p_{\rho}'(R)\psi_{k,\rho}'(R)
\end{eqnarray*}
We have then obtained the expected expression for this particular choice of vector field $V$. Similar computations enable us to recover the formula when dealing with the vector field $V$ given by $V\cdot \nu=\sin (k\cdot)$. Finally, for general $V$, one has to expand the square $(V\cdot\nu)^2$, and the computation follows exactly the same lines as before. Note that all the crossed terms of the sum (i.e. the term that do not write as squares of real numbers) vanish, by using the $L^2(\mathbb{S})$ orthogonality properties of the families $\left(\cos(k\cdot),\sin(k\cdot)\right)_{k\in \N}$.
\end{proof}
\subsubsection{Comparison principle on the family $\{\omega_{k,\rho}\}_{k\in \N^*}$}
The next result allows us to recast the ball stability issue in terms of the sign of $\omega_{1,\rho}$.
\begin{proposition}\label{Pr:Mono}
There exists $M>0$ such that, for any $\rho$ small enough,
$$
\forall k\in \N^*, \ \omega_{k,\rho}-\omega_{1,\rho}\geq -M\rho\qquad \text{and} \qquad  |\omega_{1,\rho}-\omega_{1,0}|\leq M\rho.
$$
\end{proposition}
\begin{proof}[Proof of Proposition \ref{Pr:Mono}]
Fix  $k\in \N$ and introduce $\tilde{\omega}_{k,\rho}=\omega_{k,\rho}/(\pi R)$. Using \eqref{def:omegarhok}, one computes
\begin{eqnarray*}
\tilde{\omega}_{k,\rho}-\tilde{\omega}_{1,\rho}&=& \left(-\p_{\rho}'(R)-2\phi_{\rho}'(R)\right)(\psi_{k,\rho}'(R)-\psi_{1,\rho}'(R))\\
&& -\p_{\rho}'(R)\left(\xi_{k,\rho}'(R)-\xi_{1,\rho}'(R)+\zeta_{k,\rho}'(R)-\zeta_{1,\rho}'(R)\right).
\end{eqnarray*}
We need to control each term of the expression above, which is the goal of the next results, whose proofs are postponed at the end of this section.
\begin{lemma}\label{Cl:Borne}
There exists $M>0$ and $\bar \rho>0$ such that for $\rho\in [0,\bar \rho]$, one has
$$
\max \left\{\Vert  \p_{\rho}'-\p_{0}'\Vert _{L^\infty(0,R)}, \Vert  \phi_{\rho}'\Vert _{L^\infty(0,R)} , \Vert \xi_{k,\rho}'\color{black}\Vert_\infty\color{black}\right\}\leq M\rho\quad\text{and}\quad \Vert\zeta_{k,\rho}'\color{black}\Vert_\infty\color{black}\leq M\rho^2.
$$
\end{lemma}

According to Lemma \ref{Le:Quali}, one has in particular $\p_{0}'(R)<0$.
We thus infer from Lemma~\ref{Cl:Borne} the existence of $\delta >0$ such that
$$
\min \{-\p_{\rho}'(R)-2\phi_{\rho}'(R) , -\p_{\rho}'(R)\}\geq \delta>0.
$$ 
for $\rho$ small enough. 
Furthermore, Lemma \ref{Cl:Borne} also yields easily the estimate
$$
|\zeta_{k,\rho}'(R)-\zeta_{1,\rho}'(R)|\leq M\rho^2
$$
Hence, we are done by applying the following result.
\begin{lemma}\label{lem:2estimControlOm}
There exists $\hat M>0$ and $\bar \rho>0$ such that for $\rho\in [0,\bar \rho]$, one has
\begin{equation}\label{Eq:1}
\psi_{k,\rho}'(R)-\psi_{1,\rho}'(R)\geq 0\quad \text{and} \quad \left|\xi_{k,\rho}'(R)-\xi_{1,\rho}'\right|(R)\leq \hat M\rho.
\end{equation}
\end{lemma}
Indeed, the results above lead to 
$$
\omega_{k,\rho}-\omega_{1,\rho}\geq \delta (\psi_{k}'(R)-\psi_{k,\rho}'(R)+\xi_{k,\rho}'(R)-\xi_{k,\rho}'(R))\geq 0
$$ for every $k\geq 1$ and $\rho$ small enough. 

Finally, the proof of the second inequality follows the same lines and are left to the reader.
\end{proof}

\begin{proof}[Proof of Lemma \ref{Cl:Borne}]
These convergence rates are simple consequences of elliptic regularity theory. 
Since the reasonings for each terms are similar, we only focus on the estimate of $\Vert \phi_{\rho}'\color{black}\Vert_\infty\color{black}$. Recall that $p_\rho$ solves the equation \eqref{Eq:Adj}.
Multiplying this equation by $p_\rho$, integrating by parts and using the Poincar\'e inequality yield the existence of $C>0$ such that
$$
\left(1-\rho C\Vert f'\Vert_{L^\infty(\B^*)}\right) \Vert \n p_\rho\Vert_{L^2(\B^*)}^2 \leq \rho \Vert f\Vert_{L^\infty(\B^*)}\Vert p_\rho\Vert_{L^2(\B^*)},
$$
so that $\Vert p_\rho\Vert_{W^{1,2}_0(\B^*)}$ is uniformly bounded for $\rho$ small enough. Hence, the elliptic regularity theory yields that $p_\rho$ is in fact  uniformly bounded in $W^{2,2}(\B^*)$, and there exists $\hat M>0$ such that, \color{black} defining $W^{2,2}_0(\B^*):=W^{2,2}(\B^*)\cap W^{1,2}_0(\B^*)$, \color{black}
$
\Vert p_\rho\Vert_{W^{2,2}_0(\B^*)}\leq \hat M\rho
$ 
and, since $\B^*\subset \R^2$, we get
$$
\Vert p_\rho\Vert_{L^\infty(\B^*)}\leq M\rho.
$$ 
Since $\Delta p_\rho=\rho p_\rho f'(\ur)+\rho f(\ur)$ and the right-hand side belongs to $L^{p}(\B^*)$ for all $p\geq 1$, the elliptic regularity theory yields the existence of $C>0$ such that
$$\Vert p_\rho\Vert_{W^{2,p}_0(\O)}\leq C\left(\rho \Vert p_\rho\color{black}\Vert_\infty\color{black}\Vert f'\color{black}\Vert_\infty\color{black}+\rho \Vert f\color{black}\Vert_\infty\color{black}\right)\leq M\rho$$ and using the embedding $W^{2,p}\hookrightarrow \mathscr C^{1,\alpha}$ for $p$ large enough, one finally gets 
$$
\Vert \n p_\rho\Vert_{L^\infty(\B^*)}\leq M\rho.
$$
\end{proof}

\begin{proof}[Proof of Lemma \ref{lem:2estimControlOm}]
The two estimates are proved using the maximum principle.
Let us first prove that, for any $k$ and any $\rho$ small enough, $\psi_{k,\rho}$ is non-negative on $(0,R)$.
Since, for $\rho$ small enough, $-\p_{\rho}'(R)$ is positive, and therefore $\psi_{k,\rho}(R)>0$. Since $v_k$ belongs to $W^{1,2}_0$, one has necessarily $\psi_{k,\rho}(0)=0$.
Furthermore, according to \eqref{def:psikrho},
by considering $\rho>0$ small enough so that 
$$
-\frac1{r^2}+\rho \Vert f'\color{black}\Vert_\infty\color{black}\leq -\frac1{2r^2}
$$ 
it follows that 
$$
-\frac1r(r\psi_{k,\rho}')'=c_{k,\rho}(r)\psi_{k,\rho}\quad \text{with} \quad c_{k,\rho}=-\frac{k^2}{r^2}-\rho f'(u_0)<0.
$$ 
Let us argue by contradiction, assuming that $\psi_{k,\rho}$ reaches a negative minimum at a point $r_1$. Because of the boundary condition, $r_1$ is necessarily an interior point of $(0,R)$. Then, from the equation,
$$0\geq -\psi_{k,\rho}''(r_1)=c_{k,\rho}(r_1) \psi_{k,\rho}(r_1)>0,$$ which is a contradiction. Thus there exists $\overline\rho>0$ small enough such that, for any $\rho \leq \overline \rho$ and every $k\in \N^*$, $\psi_{k,\rho}$ is non-negative on $(0,R)$.

Now, introduce $z_k=\psi_{k,\rho}-\psi_{1,\rho}$ for every $k\geq 1$ and notice that it satisfies
$$
-\frac1r(rz_k')'=\frac{1}{r^2} \psi_{1,\rho}-\frac{k^2}{r^2}\psi_{k,\rho}-\rho f'(u_0)z_k.
$$
Since $\psi_{k,\rho}$ is non-negative, it implies
$$
-\frac1r(rz_k')'\leq \left(-\frac{k^2}{r^2}-\rho f'(u_0)\right)z_k,\quad\text{and}\quad z_k(R)=z_k(0)=0.
$$ 
Up to decreasing $\bar \rho$, one may assume that for $\rho \leq \overline \rho$, $-\frac{k^2}{r^2}-\rho f'(u_0)<0$  in $(0,R)$. If $z_k$ reached a positive maximum, it would be at an interior point $r_1$, but we would have
$$0\leq -z_k''(r_1)<\left(-\frac{k^2}{r^2}-\rho f'(u_0)\right)z_k(r_1)<0.$$
Hence, one has necessarily $z_k\leq 0$ in $(0,R)$ and $z_k$ reaches a maximum at $R$, which means in particular that $z_k'(R)=\psi_{k,\rho}'(R)-\psi_{1,\rho}'(R)\geq 0$.
\end{proof}

%
%

\subsection{A further comparison result on the family $\{\omega_{k,\rho}\}_{k\in \N}$}

\color{black}
While the previous section helps us determine the sign of the sequence $\{\omega_{k,\rho}\}_{k\in \N^*}$ and thus gives us a stability criterion for the ball, we address here a more precise property, that of the optimal coercivity norm. We keep the same notation. If we assume that 
$$\forall k \in \N\,, \omega_{k,\rho}>0$$ which is guaranteed provided we have $\omega_{1,\rho}>0$ (see the next subsection \ref{sec:shapeinstab}), obtaining the $H^{1/2}$-coercivity norm is equivalent to proving that, for some constant $\ell_\rho>0$ we have 
$$
\omega_{k,\rho}\geq \ell_\rho k>0\quad \text{ for any $k$ large enough.}
$$
This property is established is the following result.

\begin{proposition}\label{Pr:H12}
There exist  $\ell_1>0$, $k_1>0$ and $M>0$ such that, for any $\rho$ small enough,
$$
\forall k\in \N^*, \qquad k\geq k_1\Longrightarrow \omega_{k,\rho}\geq\ell_1k -M\rho.
$$
As a consequence if $\omega_{k,\rho}>0$ for any $k\in \N$, then there exists a constant $\tilde\ell_0>0$ such that
$$\forall k \in \N^*\,, \omega_{k,\rho}\geq \tilde\ell_0 k.$$
\end{proposition}
\begin{proof}[Proof of Proposition \ref{Pr:H12}]
We know from Lemma \ref{Cl:Borne} that, for any $k\in \N$, there holds
\begin{eqnarray*}
\tilde{\omega}_{k,\rho}&\geq & \left(-\p_{\rho}'(R)-2\phi_{\rho}'(R)\right)\psi_{k,\rho}'(R)-M\rho.
\end{eqnarray*}

We also recall that there exists $\delta >0$ such that
$$
\min \{-\p_{\rho}'(R)-2\phi_{\rho}'(R) , -\p_{\rho}'(R)\}\geq \delta>0.
$$ 
for $\rho$ small enough. 

Let us state main ingredient of the proof.
\begin{lemma}\label{LH12}
There exist $\ell_0>$,  $\hat M>0$ and $\bar \rho>0$ such that for $\rho\in [0,\bar \rho]$, one has
\begin{equation}\label{Eq:1}
\text{ For any $k\geq k_1$, }\psi_{k,\rho}'(R)\geq \ell_0k-M\rho .
\end{equation}
\end{lemma}
According to Lemma~\ref{lem:2estimControlOm}, one has 
$\omega_{k,\rho}\geq\delta( \ell_0k-M\rho),$ yielding to the conclusion of Proposition~\ref{Pr:H12} for $\rho$ small enough.
\end{proof}

let us prove Lemma~\ref{LH12}.
\begin{proof}[Proof of Lemma~\ref{LH12}]

Observe that, for any $k\in \N$, the function $y_{k,\rho}:r\mapsto \left(\frac{r}{R}\right)^{\frac{k}{\sqrt{2}}}(-\p_\rho'(R))$ solves the ODE 
\begin{equation}
\begin{cases}
-\frac1r(ry_{k,\rho}')'=-\frac{k^2}{2r^2}y_{k,\rho}\quad \text{in }(0,1), 
\\y_{k,\rho}(R)=-\p_\rho'(R).
\end{cases}
\end{equation}
Let us consider the function $z_k:=\psi_{k,\rho}-y_{k,\rho}.$
Using the same idea as in the proof of Lemma \ref{lem:2estimControlOm}, we want to prove that 
$z_k'(R)\geq 0.$

To do so, we note that the function $z_k$ satisfies
\begin{align*}
-\frac1r(rz_k')'&= -\left(\frac{k^2}{r^2}+\rho f'(u_0)\right)\psi_{k,\rho}+\frac{k^2}{2r^2}y_k \leq -\frac{k^2}{2r^2}(\psi_{k,\rho}-y_{k,\rho}).
\end{align*} Indeed, $\psi_{k,\rho}\geq 0$ and $\frac{k^2}{r^2}+\rho f'(u_0)\geq \frac{k^2}{2r^2}$ for $\rho$ small enough, uniformly in $k$. As a consequence, we have $z_k\leq  0$. Since $z_k(R)=0$, we have $z_k'(R)\geq 0$. Since 
$$y_{k,\rho}'(R)=\frac{k}{\sqrt{2}R}(-\p_\rho'(R))$$  and since we have 
$-\p_\rho'(R)\geq \delta>0$ according to Lemma~\ref{Cl:Borne} for any $\rho>0$ small enough, one gets the desired conclusion.
\end{proof}
\color{black}

 \subsection{Shape (in)stability of $\B^*$}\label{sec:shapeinstab}
 \subsubsection{\color{black} Proof of the stability of the ball under Assumption \eqref{Eq:Hyp}}
 \color{black}
Stability under Assumption \eqref{Eq:Hyp} is well known (see \cite{DambrinePierre}) in the case where $\rho=0$. Hereafter, we recall the proof, showing by the same \color{black}method   \color{black} a stability result for $\rho>0$.
 \begin{lemma}\label{Le:Instab}
 Under assumption \eqref{Eq:Hyp}, one has $\omega_{1,0}>0$.
 \end{lemma}
 This Lemma concludes the proof of the second part of Theorem \ref{Th:Shape}. Indeed, according to Propositions \ref{Pr:Sep} and \ref{Pr:Mono} we have, for $\rho>0$ small enough, and any $k\in \N^*$, 
 $$\omega_{k,\rho}>0.$$ From Lemma \ref{LH12}, there holds, for some constant $\tilde \ell_0>0$,
 $$\forall k \in \N^*\,, \omega_{k,\rho}\geq \tilde \ell_0k.$$
\begin{eqnarray*}
 \mathcal L_{\Lambda_\rho}''(\B^*)[V,V] &\geq&\sum_{k=1}^\infty \tilde\ell_0 k\left(\alpha_k^2+\beta_k^2\right)\\
 &=&\tilde\ell_0 \Vert V\cdot \nu\Vert_{H^\frac12}^2
\end{eqnarray*}
for $\rho$ small enough.

 \begin{proof}[Proof of Lemma \ref{Le:Instab}]
 To compute $\omega_{1,0}$, recall that, for $\rho=0$, the function $\psi_{1,0}$ solves
 $$
 -\frac1r(r\psi_{1,0}')'=-\frac1{r^2}\psi_{1,0}\quad\text{and} \quad \psi_{1,0}(R)=-\p_{0}'(R),
 $$ 
 and therefore, $\psi_{1,0}(r)=-\frac{r}R \p_{0}'(R)$ for all $r\in [0,R]$, so that
  \begin{align*}
\frac{ \omega_{1,0}}{\pi R} &=-\frac{\Lambda_0}R+\frac1{2R}(\p_{0}'(R))^2+g(R)\p_{0}'(R)-\p_{0}'(R)\psi_{1,0}'(R)
 \\&=\frac1R(\p_0'(R))^2+g(R)\p_0'(R)+\frac1R(\p_0'(R))^2
 \\&= \frac2R(\p_0'(R))^2+g(R)\p_0'(R)
  \\&=-\p_{0}'(R)\left(-\frac2R\p_0'(R)-g(R)\right)
 \end{align*}
 where the expression of $\Lambda_0$ is given by \eqref{def:Lambrho}.
 Since $-R\p_{0}'(R)=\int_0^R tg(t)dt=\frac1{2\pi}\int_{\B^*}g,$ and $\varphi_0'(R)<0$, we infer that the sign of $\omega_{1,0}$ is the sign of 
 $$
-\frac2R\p_0'(R)-g(R)=\frac1{\color{black}\pi R^2\color{black}}\int_{\mathbb B^*}g-g(R),
 $$
 \color{black}
 
 and the positivity of this last quantity is  exactly Assumption  \eqref{Eq:Hyp}.
%

%
%
 The conclusion  follows.
\end{proof}
 \subsubsection{An example of instability}
 In this part, we will assume that $g$ is the constant function equal to 1, i.e. $g= 1$.
Even if the ball $\B^*$ is known to be a minimizer in the case $\rho=0$, it is a degenerate one in the sense that
 $\omega_{1,0}=0$ coming from the invariance by translations of the problem. In what follows, we exploit this fact and will construct a suitable nonlinearity $f$ such that $\B^*$ is not a local minimizer for $\rho$ small enough, in other words such that $\omega_{1,\rho}<0$.
 
 We assume without loss of generality that $R=1$ for the sake of simplicity. 
 
 \begin{lemma}\label{Le:Der}
 There holds 
 $$\omega_{1,\rho}=\frac\rho4(w_1+w_1')(1)  +\pe$$  where $w_1$ solves
  \begin{equation}\label{eq:W1}
  \left\{\begin{array}{ll}
  -(rw_1')'=-\frac1rw_1-\frac{r^2}2 f'(\p_{0})-\frac{r^2}2& \text{in }(0,1)\\
  w_1(1)=-\int_0^1 tf(\p_{0})\, dt.&
  \end{array}\right.\end{equation}
 \end{lemma}
  
 \begin{proof}[Proof of Lemma \ref{Le:Der}]
 The techniques to derive estimates follow exactly the same lines as in Lemma \ref{Cl:Borne}. 
First, we claim that 
\begin{equation}\label{eq:strasb1916}
\p_{\rho}=\p_{0}+\rho \p_1+\pe\quad \text{in }\mathscr{C}^1,
\end{equation}
where $\p_1$ satisfies
 \begin{equation}\label{def:varphi1}
\left\{\begin{array}{ll}
-\frac1r(r\p_{1}')'=-f(\p_0) & \text{ in }(0,1)\\ 
\p_1(1)=0.&
\end{array}\right.
\end{equation}
Indeed, considering the function $\delta=\p_\rho-\p_0-\rho\p_1$, one shows easily that it satisfies
$$
\left\{\begin{array}{ll}
-\frac1r(r\delta ')'=\rho (f(\p_0)-f(\p_\rho)) & \text{in }(0,1)\\
 \delta (1)=0. & 
\end{array}\right.
$$
Therefore, by mimicking the reasonings done in the proof of Lemma \ref{Cl:Borne}, involving the elliptic regularity theory, and the fact that $\Vert \p_\rho-\p_0\Vert_{W^{1,\infty}}=\operatorname{O}(\rho)$, we infer that $\Vert \delta \Vert_{\mathscr{C}^1}=O(\rho^2)$, whence the result.

\medskip

Using that $\p_\rho$ satisfies $-\frac1r(r\p_{\rho}')'+\rho f(\p_\rho)=g$ and integrating this equation yields
\begin{equation}\label{eq:varphirhoprim1}
-\p_{\rho}'(1)=\frac12-\rho \int_0^1 t f(\p_{\rho})\, dt=\frac12-\rho \int_0^1 tf(\p_{0}(t))\, dt+\pe.
\end{equation}

The Equation on $\phi_{\rho}$ reads
$$
\left\{\begin{array}{ll}
-(r\phi_{\rho}')'=r\Big(-\rho \phi_{\rho} f'(\p_{\rho})-\rho f(\p_{\rho})\Big) & \text{in }(0,1)\\
\phi_{\rho}(0)=0. & 
\end{array}\right.
$$
and according to Lemma \ref{Cl:Borne}, there holds $\Vert \phi_{\rho}\color{black}\Vert_\infty\color{black}=\operatorname{O} (\rho)$. We thus infer that
\begin{equation}\label{eq:phirhoprim1}
-\phi_{\rho}'(1)=-\rho \int_0^1 tf(\p_{0})\, dt+\pe.
\end{equation}
From \eqref{eq:varphirhoprim1} and \eqref{eq:phirhoprim1},we infer that 
\begin{eqnarray}
\Lambda_\rho&=&\frac12(\left(\p_{\rho}'(1)\right)^2-\phi_{\rho}'(1)\p_{\rho}'(1)=\frac12 \p_{0}'(1)^2-\rho \p_{0}'(1)\int_0^1 tf(\p_{0})\, dt+\rho \p_{0}'(1)\int_0^1 tf(\p_{0})\, dt+\pe \nonumber \\
&=&\frac12 \p_{0}'(1)^2+\pe.\label{eq:pLambdarho}
\end{eqnarray}

\medskip

Regarding $\psi_{1,\rho}$ and using that it satisfies \eqref{def:psikrho}, we get
$$ 
\psi_{1,\rho}(1)=-\p_{\rho}'(1)=\frac12-\rho \int_0^1 tf(\p_{0})\, dt.
$$
We then infer that $\Vert \psi_{1,\rho}+r\p_{0,\rho}'(1)\Vert_{\mathscr{C}^1} =\operatorname{O}(\rho)$.
Plugging this estimate in \eqref{def:psikrho} allows us to show that
\begin{equation}\label{strasb:1824}
\psi_{1,\rho}(r)=-\p_{0}'(1)r+\rho y_1(r)+\pe\quad  \text{in }\mathscr C^1(0,1),
\end{equation}
where $y_1$ solves
\begin{equation}\label{Eq:Y}
\left\{\begin{array}{ll}
-\left(ry_1'\right)'=-\frac1ry_1+r^2\p_{0}'(1) f'(\p_{0})& \text{in }(0,1)
\\ y_1(1)= -\int_0^1 tf(\p_{0})\, dt.&
\end{array}\right.
\end{equation}

\medskip

Regarding $\xi_{1,\rho}$ and using that it satisfies \eqref{def:xikrho}, we easily get that
$\Vert \xi_{1,\rho}\Vert_{W^{1,\infty}}=\operatorname{O}(\rho)$, according to Lemma \ref{Cl:Borne}. 
This allows us to write
\begin{equation}\label{strasb:1911}
\xi_{1,\rho}=\rho z_1+\pe \quad \text{in } \mathscr C^1(0,1)
\end{equation}
where $z_1$ satisfies
\begin{equation}\label{Eq:Z}
\left\{\begin{array}{ll}
-(rz_1')'=-\frac1rz_1+r^2\p_{0}'(1)& \text{in }(0,1)\\
z_1(1)=0.&
\end{array}
\right.
\end{equation}

\medskip

Let us now expand $\omega_{1,\rho}$ with respect to the parameter $\rho$. Recall that
 \begin{multline*}
\omega_{1,\rho}= \frac12\left(-2\psi_{1,\rho}'(1)\phi_{1,\rho}'(1)-\p_{\rho}''(1)\phi_{1,\rho}'(1)-\p_{\rho}'(R)\zeta_{1,\rho}'(R)\right.
 \\\left.
 -\xi_{1,\rho}'(1)\p_{\rho}'(1)+\Lambda_\rho+\frac1{2} (\p_{\rho}')^2+\p_{\rho}'(1)-\p_{\rho}'(1)\psi_{1,\rho}'(1)\right).\end{multline*}

Regarding the term $\p_{0,\rho}'(R)\zeta_{1,\rho}'(R)$, we know from Lemma \ref{Cl:Borne} that $\Vert \zeta_{1,\rho}'(R)\color{black}\Vert_\infty\color{black}=\pe$.

Using this estimate and plugging the expansions \eqref{eq:strasb1916}-\eqref{eq:pLambdarho}-\eqref{strasb:1824}-\eqref{strasb:1911} in the expression above yields successively
  \begin{align*}
  -2\psi_{1,\rho}'(1)\phi_{\rho}'(1)&=2\rho\p_{0}'(1)\int_0^1 tf(\p_{0})\, dt+\pe=-\rho \int_0^1 tf(\p_{0})\, dt+\pe.\\
  -\p_{\rho}''(1)\phi_{\rho}'(1)&=-\p_{0}''(1)\phi_{\rho}'(1)+\pe=\frac\rho2\int_0^1 tf(\p_{0})\, dt+\pe.
  \\-\xi_{1,\rho}'(1)\p_{\rho}'(1)&=-\p_{0}'(1)\xi_{1,\rho}'(1)+\pe=\frac\rho2z_1'(1)
  \\ \Lambda_\rho+\frac12(\p_{\rho}')^2&=\p_{0}'(1)^2-\frac\rho2\int_0^1 tf(\p_{0})\, dt+\pe=\frac14-\frac\rho2\int_0^1 tf(\p_{0})\, dt+\pe
  \\\p_{\rho}'(1)&=-\frac12+\rho \int_0^1 tf(\p_{0})\, dt+\pe
  \\-\p_{\rho}'(1)\psi_{1,\rho}'(1)&=\p_{0}'(1)^2-\rho \p_{0}'(1)y_1'(1)+\rho\p_{0}'(1)\int_0^1 tf(\p_{0})\, dt+\pe
  \\&=\frac14+\frac\rho2y_1'(1)-\frac\rho2\int_0^1 tf(\p_{0})\, dt,
  \end{align*}
   by using that $\Vert \phi_{\rho}\Vert_{W^{1,\infty}}=\operatorname{O}(\rho)$ and $\Vert \xi_{1,\rho}\Vert_{W^{1,\infty}}=\operatorname{O}(\rho)$.
This gives
  \begin{align*}
  \omega_{1,\rho}&=-\rho \int_0^1 tf(\p_{0})\, dt+\frac\rho2\int_0^1 tf(\p_{0})\, dt+\frac\rho2z_1'(1)
  \\&+\frac14-\frac\rho2\int_0^1 tf(\p_{0})\, dt-\frac12+\rho \int_0^1 tf(\p_{0})\, dt+\frac14+\frac\rho2y_1'(1)-\frac\rho2\int_0^1 tf(\p_{0})\, dt+\pe.
  \end{align*}
   
  As expected, the zero order terms cancel each other out and we get
  \begin{align*}
   \omega_{1,\rho}&=-\frac{\rho}2 \int_0^1 tf(\p_{0})\, dt+\frac\rho2z_1'(1)+\frac\rho2y_1'(1)+\pe, 
  \end{align*}
  which concludes the proof by setting $w_1=y_1+z_1$.
  
  \end{proof}

  \paragraph{Construction of the non-linearity.}
  Recall that we are looking for a non-linearity $f$ such that $\omega_{1,\rho}<0$, in other words such that $(w_1+w_1')(1)<0$ according to Lemma \ref{Le:Der}.
  To this aim, let us consider the function $w_1$ solving \eqref{eq:W1}.
Let us consider a non-negative function $f$ such that
  \begin{equation}\label{Eq:HyF}
  f'(\cdot)<-1\quad \text{ on }[0,\Vert\p_{0}\color{black}\Vert_\infty\color{black}].
  \end{equation}
  It follows that 
  $$w_1(1)=-\int_0^1 tf(\p_{0})\, dt<0.$$ 
  Besides,
  \begin{align*}
  -(rw_1')'&=-\frac1rw_1-\frac{r^2}2\left(f'(\p_{0})+1\right)\geq -\frac1rw_1
  \end{align*}
by using\eqref{Eq:HyF}. Thus $w_1$ cannot reach a local negative minimum in $(0,1)$. Moreover, by using that $w_1$ is regular ($w_1$ is the sum of two functions at least $\mathscr C^1$ according to the proof of Lemma \ref{Le:Der}) and integrating the equation above yields
$$
-rw_1'(r)+\frac{1}{2}\int_0^rs^2\left(f((\varphi_0(s))+1\right)\, ds=-\int_0^r\frac{w_1(s)}{s}\, ds
$$
 for $r>0$. The left-hand side is well-defined and it follows that so is the right-hand side, which implies that necessarily $w_1(0)=0$ (else, we would immediately reach a contradiction).
 
 Since $w_1$ cannot reach a local minimum on $(0,1)$ and since $0=w_1(0)>w_1(1)$, we get that $w_1$ is decreasing on $(1-\delta,1)$ for some $\delta>0$, ensuring that $w_1'(1)<0$. The conclusion follows.


\appendix

\begin{center}
\fbox{\large{\textsf{Appendix}}}
\end{center}
\section{Proof of Lemma \ref{claim:borne}}\label{Annexe:preuveClaim}

Recall that we want to establish a uniform (with respect to $a$ and $M$) $L^\infty$ bound on the solutions of 
\begin{equation}
\left\{\begin{array}{ll}
-\Delta u_{M,\rho,a}+M(1-a)u_{M,\rho,a}+\rho f(u_{M,\rho,a})=g,&\text{ in }D,
\\u_{M,\rho,a}\in W^{1,2}_0(\O).&\end{array}\right.\end{equation} Here, it is assumed that $g$ is non-negative.\\
Define $\phi_g$ as the solution of 
\begin{equation*}
\left\{\begin{array}{ll}
-\Delta\phi_g+\rho f(\phi_g)=g,&\text{ in }D,
\\\phi_g\in W^{1,2}_0(\O).&\end{array}\right.\end{equation*}
\color{black}Standard $W^{2,2}$ estimates in dimension 2 and 3 \color{black} show that $\phi_g$ is continuous and that 
$$
\Vert \phi_g\color{black}\Vert_\infty\color{black}<+\infty.
$$
Define $z:=\phi_g-u_{M,\rho,a}\in W^{1,2}_0(\O)$. 
We can write
$$-\Delta z +\rho \frac{f(\phi_g) - f(u_{M,\rho,a})}{\phi_g-u_{M,\rho,a}}\, z = M(1-a) u_{M,\rho,a} \geq 0$$
The generalized maximum principle, and the fact that $f$ is Lipschitz  entails that $z$ reaches its minimum on the boundary $\partial D$, so that $z$ is non-negative. Thus
$$
0\leq u_{M,\rho,a}\leq \phi_g\leq \Vert \phi_g\color{black}\Vert_\infty\color{black}<+\infty
$$
and we conclude by noting that the quantity in the right-hand side is uniformly bounded with respect to $ \rho \in [0,\underline \rho)$.
\section{Proof of Proposition \ref{Semicontinuite}}\label{Annexe:preuveSemicontinuite}
We recall that we want to establish that if $(\Omega_k)_{k\in \N}\in \mathcal O_m^\N$ $\gamma$-converges to $\O$, then 
$$J_\rho(\Omega)\leq \underset{k\to\infty}{\lim\inf}J_\rho(\Omega_k).$$
Fix such a sequence $(\Omega_k)_{k\in \N}$ that $\gamma$-converges to $\Omega$. For the sake of clarity, we drop the subscript $\rho,f$ and $g$ and define, for every $k\in \N$, $u_k\in W^{1,2}_0(D)$ the unique solution to 
\begin{equation*}
\left\{\begin{array}{ll}
-\Delta u_k +\rho f(u_k)=g\text{ in }\Omega_k,&\\
u_k \in W^{1,2}_{0}(\Omega_k),&
\\ u_k \text{ is extended by continuity  as a function in }W^{1,2}_0(D).&\end{array}\right.\end{equation*}
First note that, for any $k\in \N$, multiplying the equation by $u_k$ and integrating by parts immediately yields
\begin{align*}
\lambda_1(D)\int_D u_k^2 &=\lambda_1(D)\int_{\Omega_k}u_k^2\leq \lambda_1(\Omega_k)\int_{\Omega_k}u_k^2\leq \int_{\Omega_k}|\nabla u_k|^2
\\&\leq \Vert g\Vert_{L^2(\Omega_k)}||u||_{L^2(\Omega_k)}+\rho \Vert f\Vert_{L^\infty(\R)}|\Omega_k|^{\frac12}\Vert u_k\Vert_{L^2(\Omega_k)}.
\end{align*}
The sequence $(u_k)_{k\in \N}$ is thus uniformly bounded in $W^{1,2}_0(D)$. By the Rellich-Kondrachov Theorem, $(u_k)_{k\in \N}$ converges (up to a subsequence, strongly in $L^2(D)$ and weakly in $W^{1,2}_0(D)$) to a function $u\in W^{1,2}_0(D)$.
\\The dominated convergence theorem then yields that the sequence $(f(u_k))_{k\in \N}$ converges strongly in $L^2(D)$, to $f(u)$. Thus, the sequence $(g-f(u_k))_{k\in \N}$ converges strongly in $W^{-1,2}_0(D)$  to $g-f(u)$. Since by assumption $(\Omega_k)_{k\in \N}$  $\gamma$-converges to $\Omega$ and since the right hand term converges strongly to $g-\rho f(u)$ in $W^{-1,2}_0(D)$, it follows that $(u_k)_{k\in \N}$ converges strongly in $W^{1,2}_0(D)$ to $u$ and that $u$ solves
\begin{equation*}
\left\{\begin{array}{ll}
-\Delta u +\rho f(u)=g\text{ in }\Omega,&\\
u \in W^{1,2}_{0}(\Omega),&
\end{array}\right.\end{equation*}
\color{black} which is unique.\color{black}

This strong convergence immediately implies that 
$$J(\Omega)=\underset{k\to \infty}{\lim}J(\Omega_k),$$
thus concluding the proof of Proposition \ref{Semicontinuite}.

\section{Proof of Lemma \ref{lem:monot2}}\label{append:prooflem:monot2}
\begin{proof}[Proof of Lemma \ref{lem:monot2}]
Let us first prove that 
 $(u_{M,\rho,a})_{M\geq 0}$ is uniformly bounded in ${W^{1,2}_0(D)}$ with respect to $M$ and $\rho$. To this aim, let us multiply \eqref{eq:ua} by $u_{M,\rho,a}$ and integrate by parts. One gets
\begin{eqnarray*}
\int_D |\n u_{M,\rho,a}|^2 &\leq & \int_D |\n u_{M,\rho,a}|^2+M(1-a)u_{M,\rho,a}^2\\
&\leq & \Vert g\Vert_{\color{black}W^{-1,2}\color{black}(D)}\Vert u_{M,\rho,a}\Vert_{L^2(D)}+\rho\left(f(0)+ \Vert f\Vert_{W^{1,\infty}}\right)\Vert u_{M,\rho,a}\Vert_{L^2(D)}.
\end{eqnarray*} 
By using the Poincar\'e inequality, we infer an uniform estimate of $u_{M,\rho,a}$ in $W^{1,2}_0(D)$. According to the Rellich-Kondrachov Theorem, there exists $u^*\in W^{1,2}_0(D)$ such that, up to a subfamily, $(u_{M,\rho,a})_{M\geq 0}$ converges to $u^*$ weakly in $H^1(D)$ and strongly in $L^2(D)$. As a consequence, up to a subsequence, $(f(u_{M,\rho,a}))_{M\geq 0}$ converges to $f(u^*)$ in $L^2(D)$ by using that $f$ is Lipschitz and $(\langle g,u_{M,a_n}\rangle_{\color{black}W^{-1,2}\color{black},H^1_0})_{M\geq 0}$ converges to $\langle g,u^*\rangle_{\color{black}W^{-1,2}\color{black},H^1_0}$.
By rewriting \eqref{eq:ua} under variational form with $u=u_{M,\rho,a}$, and passing to the limit as $M\to +\infty$ after having adequately extracted subsequences, we infer that $u^*$ is the unique solution of \eqref{eq:ua}. Using $u_{M,\rho,a}$ as a test function in \eqref{eq:u} and plugging the expression yielded in the definition of $\hat J_{M,\rho}(a)$, we first obtain
$$
\hat J_{M,\rho}(a)=-\frac{\rho}{2}\int_Du_{M,\rho,a}f(u_{M,\rho,a})-\langle g,u_{M,\rho,a}\rangle_{\color{black}W^{-1,2}\color{black}(D),W^{1,2}_0(D)}
$$
so that, from the previous convergence result, we have
$$
\hat J_{M,\rho}(a)\to -\frac{\rho}{2}\int_Du^*f(u^*)-\langle g,u^*\rangle_{\color{black}W^{-1,2}\color{black}(D),W^{1,2}_0(D)}\quad \text{as }M\to +\infty.
$$
Finally, if $a=\mathbbm{1}_\O$ \color{black} and if $\O$ is a stable quasi-open set\color{black}, by multiplying \eqref{eq:ua} by $u_{M,\rho,a}$ and integrating by parts, one gets 
$$
\int_{D}|\nabla u_{M,\rho,a}|^2+M\int_{D\backslash \Omega}u_{M,\rho,a}^2=\int_D (g-\rho f(u_{M,\rho,a}))u_{M,\rho,a},
$$
and since the right-hand side is uniformly bounded with respect to $M$, we infer that $\sqrt{M}u_{M,\rho,a}$ is bounded in $L^2(D\backslash \O)$ so that \color{black} $u^*=0$ almost everywhere in $D\backslash \O$. Since $\O$ is stable, this is, by definition, equivalent to $u^*\in W^{1,2}_0(\O)$.

The conclusion follows by observing that this convergence result is indeed valid without need to extract subfamily, since the closure points of $\{u_{M,\rho,a}\}_{M>0}$ reduces to a unique element.
\color{black}\end{proof}
\color{black}
\section{Proof of Lemma \ref{Le:Ajout}}\label{Append:Ajout}
We recall that $\O$ is a stable quasi-open set if the sets 
$$H^1_0(\O):=\left\{w \in H^1(D)\,, w=0\text{ q.e in }D\backslash \O\right\}$$ and 
$$\hat H^1_0(\O):=\left\{w\in H^1(D)\,, w=0 \text{ a.e in }D\backslash \O\right\}$$ coincide.

The assumption of Lemma \ref{Le:Ajout} is  that
for $0<\rho \leq \overline \rho,$ the functional $J_\rho$ is monotonous on the set of stable quasi-open sets $\mathcal O_{m,s}(D)$:
$$\forall (\Omega_1, \Omega_2)\in (\mathcal O_{m,s}(D))^2, \quad \O_1\subset \O_2\Rightarrow J_\rho(\O_1)\geq J_\rho(\O_2).$$

Let us now prove that the functional $J_\rho$ is monotonous on $\mathcal O(D)$, namely
$$\forall (\Omega_1, \Omega_2)\in (\mathcal O(D))^2,\quad  \O_1\subset \O_2\Rightarrow J_\rho(\O_1)\geq J_\rho(\O_2).$$

We first prove that $J_\rho$ is monotonous on the set  of open sets
$$\mathcal O_{m,o}(D):=\left\{\O\in \mathcal O(D)\,, \O\text{ is open}\right\}.$$

\begin{proof}[Proof of the monotonicity on $\mathcal O_{m,o}(D)$]
We use results from \cite[Lemmas 2.3 and 2.6]{Russ2019}.

We consider two admissible open sets $\O_1\subset \Omega_2$, where $\Omega_1,\O_2\in \mathcal O_{m,o}(D)$. Let us consider, for $i=1,2$, an increasing  sequence $(\O_{i,k})_{k\in \N}$ of smooth open sets included in $\O_i$ which Hausdorff converges to $\O_i$.

We can assume that, for every $k\in \N$, $\Omega_{2,k}=\Omega_{2,k}\cup \Omega_{1,k}$, and that this sequence still Hausdorff-converges to $\O_2$. As in \cite[Proof of Point (2), Lemma 3.6]{Russ2019}, the sequence $(\Omega_{i,k})_{k\in \N}$ (strong) $\gamma$-converges to $\Omega_i$, $i=1,2$. Since the functional is continuous for the strong $\gamma-$convergence, we can pass to the limit in the inequalities
$$J_\rho(\O_{1,k})\geq J_\rho(\Omega_{2,k})$$
and obtain the required conclusion.

\end{proof}

We can now prove Lemma \ref{Le:Ajout}.

\begin{proof}[Proof of Lemma \ref{Le:Ajout}]
From a classical result recalled in \cite[Lemma 2.6]{Russ2019}, for any $\O_1\subset \Omega_2$ such that $\Omega_1,\O_2\in \mathcal O_m(D)$ and $i=1,2$, there exists a  sequence  $(\O_{i,k})_{k\in \N}$ of  open sets included in $\O_i$ that $\gamma$-converges to $\O_i$.

Up to replacing $\Omega_{2,k}$ with $\Omega_{1,k}\cup \Omega_{2,k}$, giving a new sequence that still $\gamma$-converges to $\O_2$ and is still open, we can assume that 
$$\forall k \in \N\,, \Omega_{1,k}\subset \Omega_{2,k}.$$
Since $J_\rho$ is monotonous on $\mathcal O_{m,o}$, we have, for every $k$, 
$$J_\rho(\Omega_{1,k})\geq J_\rho(\Omega_{2,k})$$ and the strong continuity for the $\gamma$-convergence of sets allows us tu pass to the limit in these inequalities, yielding the desired result.
\end{proof}
\color{black}
\section*{Acknowledgments}
We would like to warmly thank the anonymous referees for their comments, which allowed us to improve and clarify our manuscript.

\bibliographystyle{abbrv}
\bibliography{biblio}
\paragraph{Acknowledgment.}
Y. Privat and I. Mazari were partially supported by the Project ''Analysis and simulation of optimal shapes - application to lifesciences'' of the Paris City Hall. A. Henrot, I. Mazari and Y. Privat were partially supported by the ANR Project ANR-18-CE40-0013 - SHAPO on Shape Optimization.

\end{document}